\documentclass[12pt]{amsart}

\usepackage{amsmath}
\usepackage{amsthm}
\usepackage{amssymb}
\usepackage{caption}
\usepackage[curve]{xypic}
\usepackage{cite}
\usepackage{enumitem}
\usepackage{color}
\usepackage{url}
\usepackage[usenames,dvipsnames]{xcolor}
\usepackage{graphicx}
\usepackage{tikz}
\usepackage{tikz-cd}
\usepackage{hyperref} 
\usepackage{verbatim}

\setlist{itemsep=4pt, topsep=4pt}

\makeatletter

\def\chaptermark#1{}%whatever

\def\chapter{%
  \if@openright\cleardoublepage\else\clearpage\fi
  \thispagestyle{plain}\global\@topnum\z@
  \@afterindenttrue \secdef\@chapter\@schapter}

\def\@chapter[#1]#2{\refstepcounter{chapter}%
  \ifnum\c@secnumdepth<\z@ \let\@secnumber\@empty
  \else \let\@secnumber\thechapter \fi
  \typeout{\chaptername\space\@secnumber}%
  \def\@toclevel{0}%
  \ifx\chaptername\appendixname \@tocwriteb\tocappendix{chapter}{#2}%
  \else \@tocwriteb\tocchapter{chapter}{#2}\fi
  \chaptermark{#1}%
  \addtocontents{lof}{\protect\addvspace{10\p@}}%
  \addtocontents{lot}{\protect\addvspace{10\p@}}%
  \@makechapterhead{#2}\@afterheading}
\def\@schapter#1{\typeout{#1}%
  \let\@secnumber\@empty
  \def\@toclevel{0}%
  \ifx\chaptername\appendixname \@tocwriteb\tocappendix{chapter}{#1}%
  \else \@tocwriteb\tocchapter{chapter}{#1}\fi
  \chaptermark{#1}%
  \addtocontents{lof}{\protect\addvspace{10\p@}}%
  \addtocontents{lot}{\protect\addvspace{10\p@}}%
  \@makeschapterhead{#1}\@afterheading}
\newcommand\chaptername{Chapter}

\def\@makechapterhead#1{\global\topskip 7.5pc\relax
  \begingroup
  \fontsize{\@xivpt}{18}\bfseries\centering
    \ifnum\c@secnumdepth>\m@ne
      \leavevmode \hskip-\leftskip
      \rlap{\vbox to\z@{\vss
          \centerline{\normalsize\mdseries
              \uppercase\@xp{\chaptername}\enspace\thechapter}
          \vskip 3pc}}\hskip\leftskip\fi
     #1\par \endgroup
  \skip@34\p@ \advance\skip@-\normalbaselineskip
  \vskip\skip@ }
\def\@makeschapterhead#1{\global\topskip 7.5pc\relax
  \begingroup
  \fontsize{\@xivpt}{18}\bfseries\centering
  #1\par \endgroup
  \skip@34\p@ \advance\skip@-\normalbaselineskip
  \vskip\skip@ }
\def\appendix{\par
  \c@chapter\z@ \c@section\z@
  \let\chaptername\appendixname
  \def\thechapter{\@Alph\c@chapter}}

\newcounter{chapter}

\newif\if@openright

\makeatother

\makeatletter % makes "@" a normal character so that it can be used in macros
\def\@cite#1#2{{\m@th\upshape\bfseries%
[{#1\if@tempswa{\m@th\upshape\mdseries, #2}\fi}]}}
\makeatother % returns "@" to its normal state

\theoremstyle{plain}
\newtheorem{thm}{Theorem}[section]% subsection
\newtheorem{cor}[thm]{Corollary}
\newtheorem{ass}[thm]{Assumption}
\newtheorem{prop}[thm]{Proposition}
\newtheorem{lem}[thm]{Lemma}
\newtheorem{sublem}[thm]{Sublemma}
\newtheorem{con}[thm]{Convention}
\theoremstyle{definition}
\newtheorem{defn}[thm]{Definition}

\theoremstyle{remark}
\newtheorem{rem}[thm]{Remark}

\numberwithin{equation}{subsection}
\renewcommand{\bold}[1]{\medskip \noindent {\bf #1 }\nopagebreak}

\newcommand{\nc}{\newcommand}
\newcommand{\rnc}{\renewcommand}
\nc\bA{\mathbb{A}}
\nc\bB{\mathbb{B}}
\nc\bC{\mathbb{C}}
\nc\bD{\mathbb{D}}
\nc\bE{\mathbb{E}}
\nc\bF{\mathbb{F}}
\nc\bG{\mathbb{G}}
\nc\bH{\mathbb{H}}
\nc\bI{\mathbb{I}}
\nc{\bJ}{\mathbb{J}} 
\nc\bK{\mathbb{K}}
\nc\bL{\mathbb{L}}
\nc\bM{\mathbb{M}}
\nc\bN{\mathbb{N}}
\nc\bO{\mathbb{O}}
\nc\bP{\mathbb{P}}
\nc\bQ{\mathbb{Q}}
\nc\bR{\mathbb{R}}
\nc\bS{\mathbb{S}}
\nc\bT{\mathbb{T}}
\nc\bU{\mathbb{U}}
\nc\bV{\mathbb{V}}
\nc\bW{\mathbb{W}}
\nc\bY{\mathbb{Y}}
\nc\bX{\mathbb{X}}
\nc\bZ{\mathbb{Z}}
\nc\cA{\mathcal{A}}
\nc\cB{\mathcal{B}}
\nc\cC{\mathcal{C}}
\rnc\cD{\mathcal{D}}
\nc\cE{\mathcal{E}}
\nc\cF{\mathcal{F}}
\nc\cG{\mathcal{G}}
\rnc\cH{\mathcal{H}}
\nc\cI{\mathcal{I}}
\nc{\cJ}{\mathcal{J}} 
\nc\cK{\mathcal{K}}
\rnc\cL{\mathcal{L}}
\nc\cM{\mathcal{M}}
\nc\cN{\mathcal{N}}
\nc\cO{\mathcal{O}}
\nc\cP{\mathcal{P}}
\nc\cQ{\mathcal{Q}}
\rnc\cR{\mathcal{R}}
\nc\cS{\mathcal{S}}
\nc\cT{\mathcal{T}}
\nc\cU{\mathcal{U}}
\nc\cV{\mathcal{V}}
\nc\cW{\mathcal{W}}
\nc\cY{\mathcal{Y}}
\nc\cX{\mathcal{X}}
\nc\cZ{\mathcal{Z}}
\nc\bfA{\mathbf{A}}
\nc\bfB{\mathbf{B}}
\nc\bfC{\mathbf{C}}
\nc\bfD{\mathbf{D}}
\nc\bfE{\mathbf{E}}
\nc\bfF{\mathbf{F}}
\nc\bfG{\mathbf{G}}
\nc\bfH{\mathbf{H}}
\nc\bfI{\mathbf{I}}
\nc{\bfJ}{\mathbf{J}} 
\nc\bfK{\mathbf{K}}
\nc\bfL{\mathbf{L}}
\nc\bfM{\mathbf{M}}
\nc\bfN{\mathbf{N}}
\nc\bfO{\mathbf{O}}
\nc\bfP{\mathbf{P}}
\nc\bfQ{\mathbf{Q}}
\nc\bfR{\mathbf{R}}
\nc\bfS{\mathbf{S}}
\nc\bfT{\mathbf{T}}
\nc\bfU{\mathbf{U}}
\nc\bfV{\mathbf{V}}
\nc\bfW{\mathbf{W}}
\nc\bfY{\mathbf{Y}}
\nc\bfX{\mathbf{X}}
\nc\bfZ{\mathbf{Z}}

\newcommand{\bk}{{\mathbf{k}}}

\nc{\dmo}{\DeclareMathOperator}
\nc{\wt}{\widetilde}
\rnc{\Re}{\operatorname{Re}}
\rnc{\Im}{\operatorname{Im}}
\rnc{\span}{\operatorname{span}}
\dmo{\rank}{rank}
\dmo{\End}{End}
\dmo{\Hom}{Hom}
\dmo{\Jac}{Jac}
\dmo{\Id}{Id}
\dmo{\Ann}{Ann}
\dmo{\Area}{Area}
\dmo{\CP}{\bC P^1}
\dmo{\rk}{rk}
\dmo{\rel}{rel}
\dmo{\ra}{\rightarrow}

\rnc{\Col}{\operatorname{Col}}

\nc{\ColOne}{\Col_{\bfC_1}}
\nc{\ColOneX}{\ColOne(X,\omega)}

\nc{\ColTwo}{\Col_{\bfC_2}}
\nc{\ColTwoX}{\ColTwo(X,\omega)}

\nc{\ColThree}{\Col_{\bfC_3}}
\nc{\ColThreeX}{\ColThree(X,\omega)}

\nc{\ColOneTwo}{\Col_{\bfC_1, \bfC_2}}
\nc{\ColOneTwoX}{\ColOneTwo(X,\omega)}

\nc{\ColOneThree}{\Col_{\bfC_1, \bfC_3}}
\nc{\ColOneThreeX}{\ColOneThree(X,\omega)}

\nc{\MOne}{\cM_{\bfC_1}}
\nc{\MTwo}{\cM_{\bfC_2}}
\nc{\MOneTwo}{\cM_{\bfC_1, \bfC_2}}

\nc{\MThree}{\cM_{\bfC_3}}

\nc{\MOneThree}{\cM_{\bfC_1, \bfC_3}}

\dmo{\For}{\cF}

\nc{\GL}{\mathrm{GL}^+(2, \bR)}

\title[Minimal Homological Dimension]{Invariant Subvarieties of Minimal Homological Dimension, Zero Lyapunov Exponents, and Monodromy}
\author[Apisa]{Paul~Apisa}

\begin{document}
\maketitle
% removes page number from first page
\thispagestyle{empty}

%%%%%%%%%%%%%%%%%%% 
% TABLE OF CONTENTS
%%%%%%%%%%%%%%%%%%%
% allows subsections (depth 1) to be displayed in table of contents
\setcounter{tocdepth}{1} 
\tableofcontents

%%%%%%%%%%%%%%%%%%%%%%%%%%%%%%%%%%%%%%%%%%%
\section{Introduction}\label{S:intro}
%%%%%%%%%%%%%%%%%%%%%%%%%%%%%%%%%%%%%%%%%%%

Over a finite cover of the moduli space, $\cM_g$, of genus $g$ Riemann surfaces there are two well-studied vector bundles. One, denoted $\Omega \cM_g$, is a $g$-dimensional complex vector bundle whose fiber over a Riemann surface $X \in \cM_g$ is the vector space of holomorphic one-forms on $X$. The other, which we denote $H$, is the $2g$-dimensional real vector bundle whose fiber over $X$ is $H^1(X, \mathbb{R})$. This bundle is equipped with the Gauss-Manin connection. Let $H_{\mathbb{R}}^1$ denote the pullback of $H$ to $\Omega \cM_g$. 

The space $\Omega \cM_g$ is equipped with a $\mathrm{GL}(2, \mathbb{R})$ action induced from complex scalar multiplication and Teichm\"uller geodesic flow and admits an invariant stratification by specifying the number and order of zeros of the holomorphic one-forms.

Parallel transport by the Gauss-Manin connection over Teichm\"uller geodesics induces a cocycle on $H_{\mathbb{R}}^1$ called the Kontsevich-Zorich cocycle.  Zorich \cite{ZorichWinding} and Forni \cite{F} demonstrated that the Lyapunov exponents of this cocycle govern the deviations of ergodic averages along straight line flow on flat surfaces and Forni \cite{F} connected these exponents to the geometry of $\cM_g$, a topic elaborated upon in Forni-Matheus-Zorich \cite{FMZ}.

\vspace{3mm}
\noindent \textbf{Homological Dimension of an Invariant Subvariety}

Given an $\mathrm{SL}(2, \mathbb{R})$ invariant measure on a stratum of $\Omega \cM_g$, one could ask for the Lyapunov exponents of the Kontsevich-Zorich cocycle. 

By work of Eskin and Mirzakhani \cite{EM}, every $\mathrm{SL}(2, \mathbb{R})$ ergodic invariant measure  is Lebesgue measure restricted to invariant subvarieties (the fact that the support of such a measures is a subvariety and not simply a sub-orbifold is due to Filip \cite{Fi1}). Forni's criterion (see \cite{F} and \cite{Forni:Criterion}) states that for such measures the number of nonzero Lyapunov exponents of the Kontsevich-Zorich cocycle is bounded below by twice the homological dimension. 

\begin{defn}
Given an invariant subvariety $\cM$, its \emph{homological dimension} is the maximum, taken over all horizontally periodic surfaces $(X, \omega) \in \cM$, of the dimension of the span in $H^1(X, \mathbb{R})$ of the core curves of horizontal cylinders on $(X, \omega)$.
\end{defn}

While the homological dimension of an invariant subvariety $\cM$ is often difficult to compute, it is bounded below by a more tractable quantity, called the rank of $\cM$, which we define now. The tangent space of an invariant subvariety $\cM$ at a point $(X, \omega)$, where $X$ is a Riemann surface and $\omega$ a holomorphic one-form on it, can be identified with a subspace of $H^1\left( X, \Sigma; \mathbb{C} \right)$ where $\Sigma$ is the collection of zeros of $\omega$. Letting $p: H^1\left( X, \Sigma; \mathbb{C} \right) \ra H^1(X; \mathbb{C})$ be the projection to absolute cohomology, work of Avila-Eskin-M\"oller \cite{AEM} states that $p\left( T_{(X, \omega)} \cM \right)$ is a complex symplectic subspace. 

\begin{defn}
The \emph{rank} of an invariant subvariety $\cM$ is half the dimension of $p\left( T_{(X, \omega)} \cM \right)$, where $(X, \omega) \in \cM$. The \emph{rel} of $\cM$ is the dimension of the kernel of $p \restriction_{T_{(X, \omega)} \cM}$. The definitions are independent of the choice of point in $\cM$.
\end{defn}

Building on work of Eskin, Mirzakhani, and Mohammadi \cite{EMM} and Smillie-Weiss \cite{SW2}, Wright \cite{Wcyl} showed that the homological dimension of an invariant subvariety is always bounded below by its rank. This observation inspires the following definition. 

\begin{defn}
An invariant subvariety $\cM$ is said to have \emph{minimal homological dimension} if its rank and homological dimension coincide. 
\end{defn}
\begin{rem}
In Apisa-Wright \cite[Example 2.5]{ApisaWrightGeminal}, it was shown that there are infinitely many examples of invariant subvarieties of every rank and every dimension that have minimal homological dimension, although this terminology is not used there. These examples include the infinite sequence of square-tiled surfaces studied by Matheus-Yoccoz \cite{MY}. Moreover, these examples are nontrivial in the sense that if $\cM$ is one such example, then it is not \emph{full rank}, i.e. it is not the case that $\cM \subseteq \Omega \cM_g$ and $\mathrm{rank}(\cM) = g$. If $\cM$ is full rank then it automatically has minimal homological dimension and, by Mirzakhani-Wright \cite{MirWri2}, it is a component of a stratum of Abelian differentials or a hyperelliptic loci therein. 
\end{rem}

Our first theorem is a characterization of invariant subvarieties of minimal homological dimension. 

\begin{thm}\label{T:main}
If $\cM$ has minimal homological dimension then it is either a component of a stratum of Abelian differentials or a full locus of covers of the hyperelliptic locus in a stratum of Abelian differentials. 

Moreover, if $\mathrm{rank}(\cM) > 1$, then after marking all periodic points, $\cM$ is $h$-geminal.
\end{thm}
\begin{rem}
The definition of $h$-geminal first appeared in Apisa-Wright \cite[Definition 4.3]{ApisaWrightGeminal} and will be recalled in Section \ref{S:GeminalandMP}. Knowing that $\cM$ is $h$-geminal places strong constraints on the covers of which the surfaces in $\cM$ are the domain.  

The definition of periodic point best adapted to our purposes here is the one in Apisa-Wright \cite{ApisaWright}. Note that, in Theorem \ref{T:main}, when periodic points are marked, this only requires marking finitely many points by Eskin-Filip-Wright \cite[Theorem 1.5]{EFW} (see Apisa-Wright \cite[Section 4.2]{ApisaWright} for a discussion).

Finally, when $\cM$ has rank one, Theorem \ref{T:main} is equivalent to the statement that $\cM$ is a locus of torus covers. This is a consequence of Wright \cite[Theorem 1.9]{Wcyl} and the fact that loci of torus covers are precisely rank one invariant subvarieties with field of definition $\mathbb{Q}$. This enables us to exclusively consider invariant subvarieties of rank at least two in the proof of Theorem \ref{T:main}. 
\end{rem}

A tool used in the proof of Theorem \ref{T:main} is the following (note that the definition of $\cM$-parallelism is recalled in  Definition \ref{D:Cylinder}). 

\begin{thm}\label{T:MHDImpliesHomology}
$\cM$ has minimal homological dimension if and only if any two $\cM$-parallel cylinders have homologous core curves. 
\end{thm}

Theorems \ref{T:main} and \ref{T:MHDImpliesHomology} answer a question of Mirzakhani-Wright \cite[Question 1.7]{MirWri2} which asks what invariant subvarieties $\cM$ of rank at least two have the property that any two $\cM$-parallel cylinders have homologous core curves.

\vspace{6mm}
\noindent \textbf{Vanishing Lyapunov Exponents of the Kontsevich-Zorich Cocycle}

Since the rank, $r$, of an invariant subvariety, $\cM \subseteq \Omega \cM_g$, is a lower bound on its homological dimension, Forni's criterion (and the cylinder deformation theorem of Wright \cite{Wcyl}) implies that at most $2g - 2r$ Lyapunov exponents of the Kontsevich-Zorich cocycle for Lebesgue measure on $\cM$ vanish. (Note that if equality holds, then $\cM$ has minimal homological dimension). 

There are two known invariant subvarieties that are not full rank, i.e $r \ne g$, and that have  $2g-2r$ zero Lyapunov exponents. These are the Eierlegende-Wollmilchsau (independently studied by Forni \cite{Fsur} and Herrlich-Schmith\"{u}sen \cite{HS}) and the Ornithorynque (studied by Forni-Matheus \cite{FM}, see also Forni-Matheus-Zorich \cite{FMZ-Square}). Both of these subvarieties are Teichm\"uller curves and hence rank one\footnote{The projections of the Eierlegende-Wollmilchsau and Ornithorynque to $\cM_g$ are the families of compact Riemann surfaces defined by $y^d = x(x-1)(x-t)$, where $t \in \mathbb{P}^1 - \{0, 1, \infty\}$, for $d = 4$ and $d=6$ respectively. McMullen \cite[Theorem 8.3]{McM:braid} identified rigid factors in the Jacobian of these families that also explain the vanishing of $2g-2$ Lyapunov exponents.}. Building on work of M\"oller \cite{M5}, recent work of Aulicino-Norton \cite{AulicinoNorton} showed that these examples are the only rank one invariant subvarieties with $2g-2$ vanishing exponents. Via Theorem \ref{T:main}, we extend this to the following. 

\begin{thm}\label{T:main2}
Let $\cM \subseteq \Omega \cM_g$ be an invariant subvariety of rank $r$. Then $\cM$ has $2g-2r$ zero Lyapunov exponents (which is the maximum possible given the rank of $\cM$) if and only if $\cM$ is the Eierlegende-Wollmilchsau, the Ornithorynque, or full rank.
\end{thm}
\begin{rem}\label{R:T:main2}
The backwards implication is an immediate consequence of Filip \cite[Corollary 1.7 (ii)]{FiZero} (for full rank), Forni \cite{Fsur} (for the Eierlegende-Wollmilchsau), and Forni-Matheus \cite{FM} (for the Ornithorynque).

The proof will show that if $\cM$ has rank $r > 1$, is not full rank, and has minimal homological dimension then there are never $2g-2r$ zero Lyapunov exponents. This reduces Theorem \ref{T:main2} to the work of Aulicino-Norton \cite{AulicinoNorton} and M\"oller \cite{M5} mentioned above. We would like to stress that the present work relies on Aulicino-Norton for the rank one case and does not offer a new proof of their result.

To implement this proof strategy we rely on properties of geminal varieties shown in Apisa-Wright \cite{ApisaWrightGeminal} and use the non-varying phenomena and formulas for sums of Lyapunov exponents found in Eskin-Kontsevich-Zorich \cite{EKZbig}.
\end{rem}

\noindent \textbf{Monodromy}

% Let $\cM$ be an invariant subvariety. By Avila-Eskin-M\"oller \cite[Theorem 1.5]{AEM}, the Hodge bundle $H_{\mathbb{R}}^1$ over $\cM$ is semisimple in the sense that any flat bundle has a complementary flat bundle. We will let $W$ denote the complement of $p(T\cM)$ in $H_{\mathbb{R}}^1$ over $\cM$, i.e.
% % Let $\cM$ be an invariant subvariety. Let $W$ be the flat subbundle of the Hodge bundle $H_{\mathbb{R}}^1$ over $\cM$ consisting of cohomology classes symplectically orthogonal to $p(T\cM)$. Since, by Avila-Eskin-M\"oller \cite{AEM}, $p(T\cM)$ is complex-symplectic we have the following,
% \[ H_{\mathbb{R}}^1 = p(T\cM) \oplus W.\] 
% Since $p(T\cM)$ is complex-symplectic (by Avila-Eskin-M\"oller \cite{AEM}) $W$ is simply the subbundle of the Hodge bundle that is symplectically orthogonal to $p(T\cM)$.
Let $\cM$ be an invariant subvariety. Since $p(T\cM)$ is complex-symplectic (by Avila-Eskin-M\"oller \cite{AEM}) letting $W$ be the subbundle of $H_{\mathbb{R}}^1 \restriction_{\cM}$ that is symplectically orthogonal to $p(T\cM)$, implies that
\[ H_{\mathbb{R}}^1\restriction_{\cM} = p(T\cM) \oplus W.\] 
The Zariski closure of the monodromy of $p(T\cM)$ is the full symplectic group by Filip \cite[Corollary 1.7 (i)]{FiZero}. However, the monodromy of $W$ is mysterious. 

By Filip \cite[Corollary 1.7 (ii)]{FiZero}, there are no zero Lyapunov exponents in $p(T\cM)$. This implies that outside of the three exceptions listed in Theorem \ref{T:main2}, the Lyapunov exponents of the Kontsevich-Zorich cocycle restricted to $W$ are not all zero. An immediate corollary is the following.

\begin{cor}
The Zariski closure of the monodromy of the bundle $W$ over $\cM$ belongs to a compact group if and only if $\cM$ is one of the following: the Eierlegende-Wollmilchsau, the Ornithorynque, or full rank (in which case $W$ is a zero-dimensional bundle and the claim is vacuous).
\end{cor}

%%%%%%%%%%%%%%%%%%%%
% DO NOT DELETE
% THIS IS A NICE DESCRIPTION OF THE DECOMPOSITION OF H^1
%
% By Wright \cite{Wfield}, there is the following decomposition of $H^1$, over $\cM$, into flat subbundles,
% \[ H^1 = p(T\cM) \oplus \left( \bigoplus_\iota p(T\cM)^\iota \right) \oplus V\]
% where $p(T\cM)^{\iota}$ denotes Galois conjugates of $p(T\cM)$ and where $V$ is both Hodge and symplectically orthogonal to the other summands. (By Wright \cite{Wfield}, $p(T\cM)$ (and its Galois conjugates) are all simple bundles.) By Filip \cite[Corollary 1.7 (ii){FiZero}, there are no zero Lyapunov exponents in $p(T\cM) \oplus \left( \bigoplus_\iota p(T\cM)^\iota \right)$. 
%%%%%%%%%%%%%%%%%%

\vspace{3mm}
\noindent \textbf{Context}

Many authors have studied the occurrence of zero Lyapunov exponents in the spectrum of the Kontsevich-Zorich cocycle. Filip proved in \cite{FiZero} that zero Lyapunov exponents ``arise from monodromy". The problem of which invariant subvarieties have completely degenerate Lyapunov spectra was solved in Aulicino-Norton \cite{AulicinoNorton}. Further examples of invariant subvarieties with zero exponents were constructed by Forni-Matheus-Zorich \cite{FMZ-Zero} and Grivaux-Hubert \cite{GrivauxHubert-Fully}. At the other end of the spectrum, Avila-Viana \cite{AvilaViana} proved the simplicity of the Lyapunov spectrum for Masur-Veech measure on strata. 

For work on sums of Lyapunov exponents see Eskin-Kontsevich-Zorich \cite{EKZbig}. For work on the Lyapunov spectrum of Veech surfaces see Bouw-M\"oller \cite{BM}, Eskin-Kontsevich-Zorich \cite{EKZsmall} and Forni-Matheus-Zorich \cite{FMZ-Square} (for square-tiled surfaces), and Bainbridge \cite{BainbridgeThesis} (for genus two). Bainbridge's work builds on McMullen's classification of $\mathrm{SL}(2, \mathbb{R})$-invariant probability measures in genus two \cite{Mc5}.

An infinite collection of square-tiled surfaces with homological dimension one was produced by Matheus-Yoccoz \cite{MY}. McMullen \cite{McM:Modular} introduced a notion similar to that of a square-tiled surface having homological dimension one, namely square-tiled surfaces in which every cusp has rank one. These surfaces are precisely the ones for which any sequence of closed flat geodesics, whose lengths tend to infinity, equidistribute on the surface, see McMullen \cite[Theorem 9.1]{McM:Modular} for a precise statement of a stronger result.

Finally, exponents have recently been connected to the wind-tree model, see Delecroix-Hubert-Leli\`evre \cite{DelecroixHubertLelievre}.

\bold{Acknowledgments.}  The author is grateful to Alex Wright, David Aulicino, and Curt McMullen for helpful conversations and comments.  During the preparation of this paper, the  author was partially supported by NSF Postdoctoral Fellowship DMS 1803625.

%%%%%%%%%%%%%%%%%%%%%%%%%%%%%%%%%%%%%%%%%%%
%%%%%%%%%%%%%%%%%%%%%%%%%%%%%%%%%%%%%%%%%%%
%
% SECTION - Background
%
%%%%%%%%%%%%%%%%%%%%%%%%%%%%%%%%%%%%%%%%%%%
%%%%%%%%%%%%%%%%%%%%%%%%%%%%%%%%%%%%%%%%%%%
\section{Preliminaries on cylinder deformations}
%%%%%%%%%%%%%%%%%%%%%%%%%%%%%%%%%%%%%%%%%%%

In this section we will summarize a collection of results on cylinder deformations. To avoid repetition, throughout this section we will let $(X, \omega)$ be a translation surface in an invariant subvariety $\cM$. We will also let $\Sigma$ denote the singularities and marked points of $\omega$. 

\begin{defn}\label{D:Cylinder}
A \emph{cylinder} on $(X, \omega)$ is a maximal open embedded Euclidean cylinder. Two parallel cylinders on $(X, \omega)$ are said to be \emph{$\cM$-equivalent} (or sometimes \emph{$\cM$-parallel}) if they remain parallel on all surfaces in some neighborhood of $(X, \omega)$ in $\cM$ (see Wright \cite[Remark 4.2]{Wcyl} for an explanation of the formal meaning of a cylinder persisting in a neighborhood of a translation surface). Two cylinder equivalence classes are \emph{disjoint} if every cylinder in one equivalence class is disjoint from every cylinder in the other. A cylinder is called \emph{generic} if all of the saddle connections on its boundary remain parallel to the core curve of the cylinder in a neighborhood of $(X, \omega)$ in $\cM$. An equivalence class of cylinders $\bfC$ on $(X, \omega)$ in $\cM$ is called \emph{generic} if it consists of generic cylinders and if there is an open neighborhood of $(X, \omega)$ in $\cM$ on which $\bfC$ remains an equivalence class (and not just a subset of one).
\end{defn}

\begin{defn}\label{D:StandardShearAndTwistSpace}
Given a collection of disjoint cylinders $\bfC$ on $(X, \omega)$, let $\{\gamma_C\}_{C \in \bfC}$ be the oriented core curves of the cylinders in $\bfC$ and let $h_C$ denote the height of cylinder $C$. The \emph{standard deformation} is defined to be $\sigma_{\bfC} := \sum_{C \in \bfC} h_C \gamma_C^*$ where $\gamma_C^*$ is the element of $H^1(X, \Sigma; \mathbb{C})$ that sends a homology class to its oriented intersection number with $\gamma_C$. When $\bfC$ is horizontal we will also call $\sigma_{\bfC}$ the \emph{standard shear} and $i \sigma_{\bfC}$ the \emph{standard dilation}.

The \emph{twist space}, denoted $\mathrm{Twist}\left( \bfC, \cM \right)$, is the collection of all $v \in T_{(X, \omega)} \cM$ that can be written as $\sum_{C \in \bfC} a_C \gamma_C^*$ where $a_C \in \mathbb{C}$. The \emph{support of $v$} is the collection of cylinders $C$ with $a_C \ne 0$.
\end{defn}

The main theorem of Wright \cite[Theorem 1.1]{Wcyl}, called the cylinder deformation theorem, says that if $\bfC$ is an $\cM$-equivalence class on $(X, \omega)$ then $\sigma_{\bfC}$ belongs to $T_{(X, \omega)} \cM$.

Building on the cylinder deformation theorem, Mirzakhani and Wright  \cite[Theorem 1.5]{MirWri} characterized the vectors that belong to the twist space of an equivalence class $\bfC$ as follows. 

\begin{thm}\label{T:CylECTwistSpace}
Every element of $\mathrm{Twist}(\bfC, \cM)$ can be written uniquely as the sum of a  multiple of $\sigma_{\bfC}$ and an element of $\ker(p)$.
\end{thm}  

We will now prove a result about the structure of the twist space.

\begin{lem}\label{L:PerturbDisjoint}
Suppose that $\bfC_1$ and $\bfC_2$ are $\cM$-equivalence classes of cylinders on $(X, \omega)$. Let $(X', \omega')$ be a surface in $\cM$ on which the cylinders in $\bfC_i$ persist and belong to a (possibly larger) equivalence class $\bfC_i'$ of cylinders for $i \in \{1, 2\}$. If $\bfC_1$ and $\bfC_2$ are disjoint equivalence classes, then so are $\bfC_1'$ and $\bfC_2'$.
\end{lem}
\begin{proof}
Two equivalence classes $\bfD_1$ and $\bfD_2$ are disjoint if and only if the symplectic pairing of $p(\sigma_{\bfD_1})$ and $p(\sigma_{\bfD_2})$ is zero. Let $\sigma_i$ denote the standard deformation in $\bfC_i$ on $(X, \omega)$ and note that this cohomology class does not belong to $\ker(p)$. Therefore, by Theorem \ref{T:CylECTwistSpace}, on $(X', \omega')$, $p(\sigma_i)$ is collinear to $p(\sigma_{\bfC_i'})$. Since the symplectic pairing of $p(\sigma_1)$ and $p(\sigma_2)$ is zero the same holds for $p(\sigma_{\bfC_1'})$ and $p(\sigma_{\bfC_2'})$.
\end{proof}

\begin{lem}\label{L:TwistSpaceSumDecomposition}
Suppose that $(X, \omega)$ has a collection $\{ \bfC_i \}_{i=1}^n$ of pairwise disjoint equivalence classes. Set $\bfC := \bigcup_{i=1}^n \bfC_i$. Then 
\[ \mathrm{Twist}(\bfC, \cM) = \bigoplus_{i=1}^n \mathrm{Twist}(\bfC_i, \cM).\]
\end{lem}
\begin{proof}
Let $v \in \mathrm{Twist}(\bfC, \cM)$ and write it as $\sum_{i=1}^n v_i$ where $v_i = \sum_{C \in \bfC_i} a_C \gamma_C^*$ where $a_C \in \mathbb{C}$ and $\gamma_C^*$ is the Poincare dual of the core curve of $C$. By symmetry of hypotheses, it suffices to show that $v_1 \in \mathrm{Twist}(\bfC_1, \cM)$. Perturb $(X, \omega)$ to a nearby surface $(X', \omega')$ where all the cylinders in $\bfC_i$ persist and belong to a generic equivalence class $\bfC_i'$ of cylinders. By Lemma \ref{L:PerturbDisjoint}, $\{ \bfC_i' \}_{i=1}^n$ remains a collection of pairwise disjoint equivalence classes. Since $v$ pairs trivially with the core curves in $\bfC_i$, it also pairs trivially with those in $\bfC_i'$. The twist space decomposition \cite[Proposition 4.20]{ApisaWrightHighRank} implies that $v = \eta_{\bfC_1'} + \eta_{(X, \omega) \backslash \bfC_1'}$ where $\eta_{\bfC_1'} \in \mathrm{Twist}(\bfC_1', \cM)$, and $\eta_{(X,\omega)\setminus \bfC_1'}\in T_{(X, \omega)}(\cM)$ evaluates to zero on every saddle connection in $\overline{\bfC_1'}$. Moreover, this decomposition is unique, in that if $v = v' + v''$ where $v' = \sum_{C \in \bfC_1'} b_C \gamma_C^*$ where $b_C \in \mathbb{C}$ and $v''$ is a cohomology class that evaluates to zero on every saddle connection in $\overline{\bfC_1'}$, then $v' = \eta_{\bfC_1'}$ and $v'' = \eta_{(X, \omega) \backslash \bfC_1'}$. Since $v = v_1 + \left( \sum_{i=2}^n v_i \right)$ is such a decomposition (this uses that $\{ \bfC_i' \}_{i=1}^n$ is a collection of pairwise disjoint equivalence classes), $v_1 = \eta_{\bfC_1'} \in \mathrm{Twist}(\bfC_1', \cM)$ as desired. 
\end{proof}

%%%%%%%%%%%%%%%%%%%%
%
% Subsection - Cylindrical Stability
%
%%%%%%%%%%%%%%%%%%%%%

\subsection{Cylindrical Stability}

An important tool in the sequel is $\cM$-cylindrical stability, originally defined in Aulicino-Nguyen \cite[Definition 2.4]{aulicino2016rank}.

\begin{defn}
A surface $(X, \omega)$ is called \emph{$\cM$-cylindrically stable} if it is horizontally periodic and, letting $\bfC$ denote the collection of horizontal cylinders, $\mathrm{Twist}(\bfC, \cM)$ coincides with the space $\mathrm{Pres}\left( (X, \omega), \cM \right)$, defined to be the cohomology classes in $T_{(X, \omega)} \cM$ that evaluate to zero on the core curve of every horizontal cylinder. 
\end{defn}

The following properties of cylindrically stable surfaces were first shown in Wright \cite{Wcyl}. Recall that the field $\mathbf{k}(\cM)$ is defined to be the smallest extension of $\mathbb{Q}$ needed to define $\cM$ by linear equations in period coordinates. Recall too that $\mathrm{rel}(\cM) := \dim(\cM) - 2\mathrm{rank}(\cM)$.

\begin{lem}\label{L:Stable}
The following hold:
\begin{enumerate}
    \item\label{I:MaxEC} Any horizontally periodic surface in $\cM$ has at most $\mathrm{rank}(\cM) + \mathrm{rel}(\cM)$ equivalence classes of cylinders. If $\mathrm{rel}(\cM) = 0$, then any cylindrically stable surface in $\cM$ has precisely this many horizontal equivalence classes.
    \item\label{I:CCSpan} If a surface is cylindrically stable, then the horizontal core curves span a subset of $T_{(X,\omega)}(\cM)^*$  of dimension $\rank(\cM)$.
    \item\label{I:PerturbToStableQ} If $\bk(\cM)=\bQ$, every surface in $\cM$ can be perturbed to become cylindrically stable and square-tiled in such a way that all horizontal cylinders on the original surface stay horizontal in the perturbation. 
    \item\label{I:PerturbToStable} Given a collection of horizontal cylinders $\bfC$ on a surface $(X, \omega)$ in $\cM$ there is a path in $\cM$ from $(X, \omega)$ to a cylindrically stable surface $(X', \omega')$ along which the cylinders in $\bfC$ persist and remain horizontal. 
\end{enumerate}
\end{lem}
\begin{proof}
The first claim follows immediately from the cylinder deformation theorem and Apisa-Wright \cite[Lemma 7.10 (2) and (3)]{ApisaWrightHighRank}. The second and third claims are simply Apisa-Wright \cite[Lemma 7.10 (4) and (6)]{ApisaWrightHighRank}. We sketch a proof of the final claim following Wright \cite{Wcyl}.

The method for producing $(X', \omega')$ is iterative and proceeds as follows. Begin by applying Smillie-Weiss \cite[Cor. 6]{SW2} to produce a horizontally periodic surface $(X_1, \omega_1)$ in the horocycle orbit closure of $(X, \omega)$ on which the cylinders in $\bfC$ persist (it is possible to connect $(X, \omega)$ to $(X_1, \omega_1)$ by following the horocycle flow of $(X, \omega)$ until coming arbitrarily close to $(X_1, \omega_1)$ and then perturbing). Let $\bfC_1$ denote the collection of horizontal cylinders on $(X_1, \omega_1)$. If $\mathrm{Twist}\left( \bfC_1, \cM \right) = \mathrm{Pres}\left( (X_1, \omega_1), \cM \right)$, then we are done. 

If we are not done at stage $n$, then let $v$ be a purely imaginary element of $\mathrm{Pres}\left( (X_n, \omega_n), \cM \right) -  \mathrm{Twist}\left( \bfC_n, \cM \right)$. By \cite[Sublemma 8.7]{Wcyl}, if we travel from $(X_n, \omega_n)$ in $\cM$ any sufficiently small amount in the direction of the tangent vector $v$ we produce a surface $(X_n', \omega_n')$ where the cylinders in $\bfC_n$ persist and remain horizontal, but do not cover all of the $(X_n', \omega_n')$. Applying Smillie-Weiss \cite[Cor. 6]{SW2} again we produce a horizontally periodic surface $(X_{n+1}, \omega_{n+1})$ in the horocycle orbit closure of $(X_n, \omega_n)$ on which the cylinders in $\bfC_n$ persist and remain horizontal, but are a strict subset of the horizontal cylinders $\bfC_{n+1}$. When the process terminates we have found the desired surface. The previous two paragraphs are essentially contained in \cite[Proof of Lemma 8.6]{Wcyl}.
\end{proof}

%%%%%%%%%%%%%%%%%%%%
%
% BOUNDARY
%
%%%%%%%%%%%%%%%%%%%%

\subsection{The Mirzakhani-Wright Partial Compactification}

In Mirzakhani-Wright \cite{MirWri}, a partial compactification was constructed for every invariant subvariety $\cM$. The main feature is the boundary tangent formula, which we state now. 

Suppose that $(X_n, \omega_n)$ is a sequence of translation surfaces that converge in the Mirzakhani-Wright partial compactification to a surface $(X_\infty, \omega_\infty)$ in a component $\cM'$ of the boundary. By Mirzakhani-Wright \cite[Proposition  2.4]{MirWri} there is a map, called the \emph{collapse map}, $f_n: (X_n, \omega_n) \ra (X_\infty, \omega_\infty)$ that takes the singular set of $\omega_n$ to the singular set of $\omega_\infty$. Given the collapse map, the collection of $\emph{vanishing cycles}\ V_n$, is defined to be the kernel of $(f_n)_*: H_1(X_n, \Sigma_n) \ra H_1(X_\infty, \Sigma_\infty)$ where $\Sigma_i$ denotes the singular set of $\omega_i$. 

The following is due to Mirzakhani-Wright \cite[Theorem 2.9]{MirWri} when $(X_\infty, \omega_\infty)$ is connected and Chen-Wright \cite[Theorem 1.2 (see also Proposition 2.5)]{ChenWright} when $(X_\infty, \omega_\infty)$ is disconnected.

\begin{thm}\label{T:BoundaryTangent2}
For large enough $n$, $f_n^*$ induces an isomorphism, 
\[ T_{(X_\infty, \omega_\infty)} \cM' \cong  T_{(X_n, \omega_n)} \cM \cap \mathrm{Ann}(V_n) \]
where $\mathrm{Ann}(V_n)$ denotes the cohomology classes in $H^1(X_n, \Sigma_n)$ that evaluate to zero on the elements of $V_n$.
\end{thm}

The following is the form in which we will primarily use the previous result. 

\begin{thm}\label{T:CylinderBoundaryTangentFormula}
Suppose that $(X_n, \omega_n)$ is a sequence in $\cM$ converging to $(X_\infty, \omega_\infty)$ in a component $\cM'$ of the boundary of $\cM$. Let $f_n: (X_n, \omega_n) \ra (X_\infty, \omega_\infty)$ be the collapse maps. 
\begin{enumerate}
    \item\label{I:boundary:FindingCn} If $C$ is a cylinder on $(X_\infty, \omega_\infty)$, then for sufficiently large $n$, there is a cylinder $C_n$ whose modulus, height, and circumference converge to that of $C$ and so $f_n(C_n)$ is homotopic to $C$.
    \item\label{I:boundary:CnConvergence} If $C_n$ is a cylinder on $(X_n, \omega_n)$ whose height and circumference are bounded below and above, then there is a cylinder $C$ on $(X_\infty, \omega_\infty)$ so that $f_n(C_n)$ is homotopic to $C$ for sufficiently large $n$. 
    \item\label{I:boundary:TwistIsom} Suppose that $\bfC$ is a collection of cylinders on $(X_\infty, \omega_\infty)$. For sufficiently large $n$, let $f_n^*(\bfC)$ be the cylinders on $(X_n, \omega_n)$ produced by \eqref{I:boundary:FindingCn} corresponding to the cylinders in $\bfC$. Then $f_n^*$ induces an isomorphism from $\mathrm{Twist}(\bfC, \cM')$ to $\mathrm{Twist}(f_n^*(\bfC), \cM)$.
\end{enumerate}
\end{thm}
\begin{proof}
The first two claims are Mirzakhani-Wright \cite[Lemma 2.15]{MirWri}. Given a cylinder $C$ let $\gamma_C^*$ denote the cohomology class that records intersections with the core curve of $C$. For the final claim, notice that, if $\cH$ is the stratum containing $(X_\infty, \omega_\infty)$, then $\gamma_C^* \in T_{(X_\infty, \omega_\infty)} \cH$. Letting $C_n$ be as in \eqref{I:boundary:FindingCn}, it follows from Theorem \ref{T:BoundaryTangent2} that $\gamma_{C_n}^* \in \mathrm{Ann}(V_n)$. In particular, $\mathrm{Twist}(f_n^*(\bfC), \cM) \subseteq \mathrm{Ann}(V_n)$, so the result follows by Theorem \ref{T:BoundaryTangent2}.
%
% The final claim remains to be shown. Let $\gamma$ be the core curve of a cylinder in $\bfC$. For sufficiently large $n$, $\gamma$ does not belong to $U_n$. Let $v_n$ be any vanishing cycle for $f_n$ which we will represent as a weighted sum $\sum n_i \gamma_i$ where $\gamma_i$ is smooth and has transverse intersection with $\gamma$. Notice that the intersection number of $v_n$ and $g_n(\gamma)$, is the same as the intersection number of $f_n(v_n)$ and $\gamma$, since $g_n$ is a diffeomorphism onto its image. Since $f_n(v_n)$ is nullhomologous, $v_n$ and $\gamma$ have trivial intersection. Since every element of $\mathrm{Twist}(g_n(\bfC), \cM)$ is a linear combination of Poincare duals with core curves of $g_n(\bfC)$, we see that every element of $\mathrm{Twist}(g_n(\bfC), \cM)$ belongs to $\mathrm{Ann}(V_n)$. Therefore, by Theorem \ref{T:BoundaryTangent2}, $f_n^*$ induces an isomorphism from $\mathrm{Twist}(\bfC, \cM')$ to $\mathrm{Twist}(g_n(\bfC), \cM)$.
\end{proof}

In general it can be cumbersome to establish convergence in the Mirzakhani-Wright partial compactification, but it is straightforward in the following case.

\begin{defn}\label{D:Collapse}
Suppose that $\bfC$ is an equivalence class of generic cylinders on a surface $(X, \omega)$ in an invariant subvariety $\cM$. Let $v \in \mathrm{Twist}(\bfC, \cM)$ so that there is a number $t_v > 0$ so that the \emph{collapse path}, i.e. $(X, \omega) + tv$ for $t \in [0, t_v)$, diverges as $t$ approaches $t_v$ and so that each cylinder in $\bfC$ persists along the collapse path. By Apisa-Wright \cite[Lemma 4.9]{ApisaWrightHighRank}, this sequence converges in the Mirzakhani-Wright partial compactification to a surface $\Col_v(X, \omega)$ in a component $\cM_v$ of the boundary of $\cM$. We will let $\bfC_v$ denote the cylinders in $\bfC$ whose height go to zero along the collapse path. 

Similarly, when $\bfC$ is collection of horizontal cylinders with $\sigma_{\bfC} \in T_{(X, \omega)} \cM$, we will define $\Col_{\bfC}(X, \omega)$ to be the limit of the path $(X, \omega) - t(i\sigma_{\bfC})$ for $t \in [0, 1)$ regardless of whether or not this path diverges. The surface $\Col_{\bfC}(X, \omega)$ can be thought of as the result of vertically collapsing $\bfC$. Moreover, this definition makes sense for any collection of cylinders $\bfC$ (by rotating the surface so the cylinders in $\bfC$ are horizontal, collapsing, and then rotating back). When $\bfC$ contains a vertical saddle connection, $\Col_{\bfC}(X, \omega) = \Col_{v}(X, \omega)$ for $v = -i \sigma_{\bfC}$ and we will let $\cM_{\bfC}$ denote $\cM_v$.
\end{defn}

In the sequel we will rely on the general theory of cylinder degenerations developed in Apisa-Wright \cite[Section 4]{ApisaWrightHighRank}. We begin by making the following definitions.

\begin{defn}\label{D:InvolvedWithRel}
Suppose that $\bfC$ is an equivalence class cylinders on a surface $(X, \omega)$ in an invariant subvariety $\cM$. Say that $\bfC$ is \emph{involved with rel} if some vector in $\ker(p) \cap T_{(X, \omega)}(\cM)$ evaluates non-trivially on a cross curve of a cylinder in $\bfC$. 

Let $v \in \mathrm{Twist}(\bfC, \cM)$. Say that $v$ is \emph{typical} if all the cylinders in $\bfC_v$ have ratios of heights that are constant on a neighborhood of $(X, \omega)$ in $\cM$. 
\end{defn}

\begin{rem}\label{R:TypicalOrthogonalConstruction}
Note that if $\bfC$ is horizontal and $v$ is typical, then after perhaps shearing the cylinders in $\bfC$, the imaginary part of $v$ can be seen to define a typical cylinder degeneration. 
\end{rem}

\begin{thm}\label{T:Typical}
Suppose that $\bfC$ is a generic equivalence class of cylinders on a surface $(X, \omega)$ in an invariant subvariety $\cM$. Then there is a typical cylinder degeneration $v \in \mathrm{Twist}(\bfC, \cM)$ and for any such $v$, $\dim(\cM_v) = \dim(\cM) - 1$ and the cylinders in $\Col_v(\bfC)$ remain generic. Moreover, if $\bfC$ is involved with rel, then $\mathrm{rank}(\cM_v) = \mathrm{rank}(\cM)$. Otherwise, $\mathrm{rank}(\cM_v) = \mathrm{rank}(\cM)-1$. 
\end{thm}
\begin{proof}
Every claim except the final one is contained in Apisa-Wright \cite[Lemmas 6.5, 11.2, and 11.4]{ApisaWrightHighRank}. Suppose now that $\bfC$ is not involved with rel. This implies that the twist space of $\bfC$ is spanned by $\sigma_{\bfC}$ (by Theorem \ref{T:CylECTwistSpace}). In particular, there is a saddle connection $s$ contained in $\bfC$ whose length goes to zero on $\Col_v(X, \omega)$. Since $\cM_v$ has codimension one, every element of $(T_{(X, \omega)} \cM)^*$ generated by a vanishing cycle is collinear to $s$. Since $\bfC$ is not involved with rel, $\mathrm{rank}(\cM_v) \leq \mathrm{rank}(\cM)-1$ by Apisa-Wright \cite[Lemma 3.8]{ApisaWrightHighRank}. Equality must hold by Apisa-Wright \cite[Corollary 3.7]{ApisaWrightHighRank}.
\end{proof}

To close this subsection, we discuss the way in which translation covers interact with the boundary. Recall that a \emph{translation cover} $f: (X, \omega) \ra (Y, \eta)$ is a holomorphic branched cover $f: X \ra Y$ so that $f^* \eta = \omega$. 

\begin{con}
For translation covers the image of any marked point is a marked point and all branch points are marked. 
\end{con}

The following result morally says that a translation cover can be ``collapsed to a translation cover on the boundary".

\begin{lem}[Lemma 2.2 in Apisa-Wright \cite{ApisaWrightDiamonds}]\label{L:DefinitionCol(f)}
Suppose that $f: (X, \omega) \ra (Y, \eta)$ is a translation covering. Let $\bfC\subset (X,\omega)$ be a collection of parallel cylinders such that $f^{-1}(f(\overline\bfC)) = \overline\bfC$ and $\overline\bfC \neq (X,\omega)$. Then there is a translation cover $$\Col_{\bfC}(f): \Col_{\bfC}(X, \omega) \ra \Col_{f(\bfC)}(Y, \eta)$$ of the same degree.
\end{lem}

% \subsection{Multicomponent Boundaries and Primality}

We now recall terminology that will be used when surfaces in the boundary of $\cM$ have multiple components. 

\begin{defn}\label{D:Prime}
Suppose that $\cM'$ is an invariant subvariety in a product of strata $\cH_1 \times \hdots \cH_n$. $\cM'$ is not \emph{prime} if, after perhaps re-indexing, there is a constant $1 \leq k < n$, so that $\cM' = \cM'' \times \cM'''$ where $\cM''$ is an invariant subvariety in $\cH_1 \times \hdots \times \cH_k$ and $\cM'''$ is an invariant subvariety in $\cH_{k+1} \times \hdots \times \cH_n$. By Chen-Wright \cite[Lemma 7.10]{ChenWright}, $\cM'$ has a unique prime decomposition, i.e. $\cM'$ can be represented as a product of prime invariant subvarieties.
\end{defn}

\section{Basics about minimal homological dimension}

\begin{proof}[Proof of Theorem \ref{T:MHDImpliesHomology}:]
Suppose first that $\cM$ has minimal homological dimension. Let $\bfC$ be an equivalence class of cylinders on a surface $(X, \omega)$ in $\cM$. By applying Lemma \ref{L:Stable} \eqref{I:PerturbToStable}, we may suppose without loss of generality that $(X, \omega)$ is cylindrically stable and that $\bfC$ is contained in an equivalence class $\bfC_1$ of horizontal cylinders (applying the lemma may have caused $\bfC$ to become a proper subset of an equivalence class of cylinders). By Lemma \ref{L:Stable} \eqref{I:CCSpan}, the horizontal core curves of $(X, \omega)$ span a subset of $T_{(X,\omega)}(\cM)^*$  of dimension $\rank(\cM)$. Since any rel vector evaluates to zero on these core curves, in fact the horizontal core curves on $(X, \omega)$ span a subset of $p(T_{(X,\omega)}(\cM))^*$  of dimension $\rank(\cM)$.

Let $W$ be the span of the horizontal core curves of $(X, \omega)$ in $H_1(X)$. By definition of minimal homological dimension, $\dim W \leq \mathrm{rank}(\cM)$. Let $\phi: H_1(X) \ra p\left( T_{(X, \omega)} \cM \right)^*$ be the linear map given by associating to each homology class its corresponding functional on cohomology. Since $\dim \phi(W) = \mathrm{rank}(\cM)$, $\phi \restriction_W$ is an isomorphism onto its image. By \cite[Lemma 4.7]{Wcyl} any two core curves of a cylinder in $\bfC$ define collinear elements of $p\left( T_{(X, \omega)} \cM \right)^*$ and hence the core curves must be collinear in $H_1(X)$, in particular any two core curves of cylinders in $\bfC$ are homologous.

Now suppose that for every surface $(Y, \eta)$ in $\cM$ and every equivalence class $\bfC$ of cylinders on $(Y, \eta)$, the core curves of cylinders in $\bfC$ are homologous. Let $(X, \omega)$ be a horizontally periodic surface in $\cM$. Letting $W$ be the span of the horizontal cylinder core curves in $H_1(X)$, we want to show that $\dim W \leq \mathrm{rank}(\cM)$. By Poincare duality, it suffices to show that $\mathrm{PD}(W)$, the span of the Poincare duals of the horizontal core curves, has dimension at most $\mathrm{rank}(\cM)$. If $\bfC$ is an equivalence class of horizontal cylinders and $\gamma$ is the core curve of a cylinder in $\bfC$, then its Poincare dual $\gamma^*$ is collinear to $\sigma_{\bfC}$ since all the cylinders in $\bfC$ have homologous core curves. This shows that $\mathrm{PD}(W)$ is contained in $p(T_{(X, \omega)} \cM)$. Moreover, $\mathrm{PD}(W)$ is Lagrangian since $W$ is Lagrangian and so $\dim W = \dim \mathrm{PD}(W) \leq \mathrm{rank}(\cM)$ as desired.  
\end{proof}

\begin{cor}\label{C:FieldDefinition}
If $\cM$ has minimal homological dimension, then $\mathbf{k}(\cM) = \mathbb{Q}$.
\end{cor}
\begin{proof}
Given an equivalence class of cylinders $\bfC$, the ratio of lengths of core curves of cylinders in $\bfC$ is one since the core curves are all pairwise homologous. The claim now holds by Wright \cite[Theorem 7.1]{Wcyl}, which states that $\mathbf{k}(\cM)$ is the smallest field extension of $\mathbb{Q}$ containing the ratios of circumferences of cylinders in $\bfC$.
\end{proof}

\begin{lem}\label{L:MHD-Boundary}
Let $\cM'$ be a component of the boundary of an invariant subvariety $\cM$ containing connected surfaces. If $\cM$ has minimal homological dimension, then so does $\cM'$.
\end{lem}
\begin{proof}
Let $\cM'$ be a component of the boundary of $\cM$. Let $(X_n, \omega_n)$ be a sequence of surfaces in $\cM$ that converge to a connected surface $(Y, \eta)$ in $\cM'$. By Chen-Wright \cite[Proposition 2.5]{ChenWright} (see also Mirzakhani-Wright \cite[Proposition 2.6]{MirWri}), we may suppose without loss of generality that $(Y, \eta)$ has dense orbit in $\cM'$.

Let $(Z, \zeta)$ be any surface in $\cM'$ and let $\bfD$ be an equivalence class of cylinders on it. By Theorem \ref{T:MHDImpliesHomology}, we wish to show that the core curves of cylinders in $\bfD$ are pairwise homologous. The cylinders in $\bfD$ persist in a neighborhood $U$ of $(Z, \zeta)$ in $\cM'$. Since $(Y, \eta)$ has dense orbit in $\cM'$ there is an element $g$ of $\mathrm{GL}(2, \mathbb{R})$ so that $g \cdot (Y, \eta) \in U$. Let $\bfB$ denote the cylinders on $(Y, \eta)$ equivalent to the ones in $g^{-1}(\bfD)$ (notice that moving from $(Z, \zeta)$ to $g\cdot (Y, \eta)$ could have created new cylinders equivalent to those in $\bfD$). It suffices to show that the core curves of cylinders in $\bfB$ are pairwise homologous.

For sufficiently large $n$, $\bfB$ corresponds to a collection of cylinders $\bfB_n$ on $(X_n, \omega_n)$ (by Theorem \ref{T:CylinderBoundaryTangentFormula} \eqref{I:boundary:FindingCn}) and there is an isomorphism
\[ f^*: T_{(Y, \eta)} \cM' \ra \mathrm{Ann}(V_n) \cap T_{(X_n, \omega_n)} \cM \]
where $f$ is the collapse map, the isomorphism is the one in Theorem \ref{T:BoundaryTangent2}, and $V_n$ is the collection of vanishing cycles on $(X_n, \omega_n)$. Fix such a sufficiently large $n$ and suppose that the cylinders in $\bfB_n$ belong to $\cM$-equivalence classes $\bfC_1, \hdots, \bfC_m$. Let $\sigma$ denote $f^* \sigma_{\bfB}$. This element belongs to $T_{(X_n, \omega_n)} \cM$ by the cylinder deformation theorem and Theorem \ref{T:CylinderBoundaryTangentFormula} \eqref{I:boundary:TwistIsom}. 

Write $\sigma = \sum_{i=1}^m \sigma_i$ where $\sigma_i = \sum_{C \in \bfB_n \cap \bfC_i} a_C \gamma_C^*$ where $a_C \in \mathbb{C}$. For $i \ne j$, notice that $\bfC_i'$ and $\bfC_j'$ are disjoint since both collections of cylinders consists of cylinders with homologous core curves and since the cylinders in $\bfB_n \cap \bfC_i$ are disjoint from those in $\bfB_n \cap \bfC_j$. By Lemma \ref{L:TwistSpaceSumDecomposition}, it follows that $\sigma_i \in T_{(X_n, \omega_n)} \cM$ for all $i$.

By Theorem \ref{T:CylinderBoundaryTangentFormula} \eqref{I:boundary:TwistIsom}, $\sigma_i \in \mathrm{Ann}(V_n)$ for all $i$. Let $\bfA_i$ be the cylinder homotopic to the cylinders in $f\left( \bfB_n \cap \bfC_i \right)$. Note that the cylinders in $\bfA_i$ have homologous core curves since this is true of the cylinders in $\bfB_n \cap \bfC_i$. By Theorem \ref{T:CylinderBoundaryTangentFormula} \eqref{I:boundary:TwistIsom}, $\sigma_{\bfA_{i}} = \left( f^* \right)^{-1}\left( \sigma_i \right)$ is an elements of $T_{(Y, \eta)} \cM'$. By Theorem \ref{T:CylECTwistSpace}, for any $i$, there is a nonzero constant $a_i$ so that $\sigma_{\bfB} - a_i \sigma_{\bfA_{i}}$ is zero in absolute cohomology. In particular, this implies that $p(\sigma_{\bfA_i})$ is collinear to $p(\sigma_{\bfA_j})$ for any $i$ and $j$. Since the core curves of cylinders in $\bfA_{i}$ are pairwise homologous, this shows that the core curves of cylinders in $\bfB$ are pairwise homologous as desired.
\end{proof}

For the following it will be useful to recall the notation for cylinder degenerations established in Definition \ref{D:Collapse}.

\begin{cor}\label{C:MHD-Boundary:Connected}
Let $\bfC$ be an equivalence class of generic cylinders on $(X, \omega)$ in an invariant subvariety $\cM$ of minimal homological dimension. Suppose that $v \in \mathrm{Twist}(\bfC, \cM)$ defines a cylinder degeneration and that $\overline{\bfC} \ne (X, \omega)$. Then $\Col_v(X, \omega)$ is connected and $\cM_v$ has minimal homological dimension. 
\end{cor}
\begin{proof}
By Lemma \ref{L:MHD-Boundary} it suffices to show that $\Col_v(X, \omega)$ is connected. Suppose not in order to derive a contradiction. Since $\overline{\bfC} \ne (X, \omega)$, Smillie-Weiss \cite[Corollary 6]{SW2} implies that there is a cylinder $D$ in the complement of $\bfC$ that is disjoint from $\overline{\bfC}$. Let $\bfD$ be its equivalence class. By the cylinder proportion theorem (Nguyen-Wright \cite[Proposition 3.2]{NW}), $\bfC$ and $\bfD$ are disjoint.

By Apisa-Wright \cite[Lemma 9.1]{ApisaWrightHighRank}, $\cM_v$ is prime. By Chen-Wright \cite[Theorem 1.3]{ChenWright}, since $\cM_v$ is prime, any deformation of $\Col_v(X, \omega)$ that remains in $\cM_v$ and changes the period of some absolute cycle on one component of $\Col_v(X, \omega)$ changes the periods of absolute cycles on every component of $\Col_v(X, \omega)$.

Since $\bfD$ is disjoint from $\bfC$, it follows that $\sigma_{\bfD}$ is an element of $T_{\Col_v(X, \omega)} \cM_v$ (for instance by Theorem \ref{T:CylinderBoundaryTangentFormula} \eqref{I:boundary:TwistIsom}). Notice that along the path $\Col_v(X, \omega) + t \sigma_{\bfD}$ for $t$ a small real number, the periods of absolute cycles change on the components of $\Col_v(X, \omega)$ that contain a cylinder from $\bfD$. This shows that a cylinder from $\bfD$ persists on every component of $\Col_v(X, \omega)$. 

However, since $\Col_v(X, \omega)$ has multiple components it is possible to find a simple closed curve on one component of $\Col_v(X, \omega)$ that intersects the core curve of a cylinder in $\bfD$ exactly once. This curve remains a simple closed curve on $(X, \omega)$ and intersects the core curve of one cylinder in $\bfD$ exactly once, but does not intersect all the core curves of the cylinders in $\bfD$. This contradicts Theorem \ref{T:MHDImpliesHomology}, which states that all core curves of cylinders in $\bfD$ are homologous. 
\end{proof}

%%%%%%%%%%%%%%%%%%%%%%%%
%%%%%%%%%%%%%%%%%%%%%%%%
%
%
% Invariant Subvarieties of Rel Zero
%
%
%%%%%%%%%%%%%%%%%%%%%%%%

\section{Cylinder rigid subvarieties}

In this section we introduce a family of invariant subvarieties that is stable under degeneration and that includes every invariant subvariety of rel zero that consists of connected surfaces. 

\begin{defn}\label{D:CS-AIS}
Suppose that $\cM$ is an invariant subvariety (possibly in a stratum of disconnected surfaces). A collection of cylinders $\bfC$ is called a \emph{subequivalence class} if the cylinders in $\bfC$ are all $\cM$-equivalent to each other and $\bfC$ is a minimal collection of cylinders so that $\sigma_{\bfC}$ belongs to $T_{(X, \omega)} \cM$.

We will say that $\cM$ is a \emph{CR invariant subvariety}, where CR stands for ``cylinder rigid" if there are finite sets $S_1 \subseteq \mathbb{Q}$ and $S_2 \subseteq \mathbb{R}$ so that the following holds:
\begin{enumerate}
    \item\label{I:SC-Partition} For each equivalence class $\bfC$ of cylinders on a surface $(X, \omega)$ in $\cM$, $\sigma_{\bfC} \in T_{(X, \omega)} \cM$. 
    \item\label{I:SC-Persistence} Subequivalent cylinders remain subequivalent as long as they persist. 
    \item\label{I:SC-Ratios} The ratio of moduli (resp. circumferences) of subequivalent cylinders belongs to $S_1$ (resp. $S_2$).
\end{enumerate}
Note that \eqref{I:SC-Partition} implies that every equivalence class on a surface in $\cM$ can be partitioned into subequivalence classes.
\end{defn}

Note too that \eqref{I:SC-Partition} is the statement of the cylinder deformation theorem. The reason we must include it as a condition is entirely because the cylinder deformation theorem is not known to hold when $(X, \omega)$ is disconnected.

The following is entirely due to Mirzakhani-Wright \cite{MirWri}.

\begin{lem}\label{L:ModuliRelZero}
If $\cM$ has rel zero and any surface in $\cM$ is connected, then $\cM$ is a CR invariant subvariety.
\end{lem}
\begin{proof}
For equivalent cylinders, \eqref{I:SC-Ratios} holds for moduli, by \cite[Corollary 1.6]{MirWri} and, for circumferences, by the cylinder finiteness theorem  \cite[Theorem 1.4]{MirWri}. Since the ratio of heights of any pair of equivalent cylinders cannot be changed by deforming the surface, it follows that equivalence classes are subequivalence classes. Since equivalent cylinders remains equivalent as long as they persist \eqref{I:SC-Persistence} holds. Condition \eqref{I:SC-Partition} is the cylinder deformation theorem of Wright \cite{Wcyl}. 
\end{proof}

\subsection{Preliminary Results}

\begin{lem}\label{L:NoNewECsInRelZero}
Let $\bfC$ be a subequivalence class of cylinders on a surface $(X, \omega)$ in a CR subvariety $\cM$. On any surface in $\cM$ where the cylinders in $\bfC$ persist, they remain a subequivalence class.
\end{lem}
\begin{proof}
Suppose that the claim fails and that $(X', \omega')$ is a surface where the cylinders in $\bfC$ are a subset of a larger subequivalence class $\bfC'$. By definition of subequivalence class, $\sigma_{\bfC} \in T_{(X, \omega)} \cM$ and $\sigma_{\bfC'} \in T_{(X', \omega')} \cM$. But then the cylinders in $\bfC$ have their moduli change along the path $(X', \omega') + t \sigma_{\bfC}$ whereas those in $\bfC' - \bfC$ have constant moduli contradicting Definition \ref{D:CS-AIS} \eqref{I:SC-Persistence} and \eqref{I:SC-Ratios}.
%
% By assumption there is a component $(Y, \eta)$ of $(X', \omega')$ where there are more cylinders in $\bfC'$ than $\bfC$. By Theorem \ref{T:CylECTwistSpace}, there is a nonzero constant $a$ so that the restriction of $\sigma_{\bfC'} - a \sigma_{\bfC}$ is a nonzero rel vector on
%
% Suppose without loss of generality that the cylinders in $\bfC$ are horizontal and that $\sigma$ is the standard shear on $\bfC$ on $(X, \omega)$. If there is a surface $(X', \omega') \in U$ on which $\bfC$ is not a subequivalence class, then there is a subequivalence class of cylinders $\bfC'$ strictly containing those in $\bfC$. But then $\sigma$ and $\sigma_{\bfC'}$ are two elements of $\mathrm{Twist}\left( \bfC', \cM \right)$ that are not multiples of each other. Hence there is a purely imaginary linear combination $v$ of these two elements, so that flowing in the direction of $v$ changes the heights of some cylinders in $\bfC'$, but not others, contradicting Definition \ref{D:CS-AIS} \eqref{I:SC-Persistence} and \eqref{I:SC-Ratios}.
\end{proof}

The following argument is often described as an ``overcollapsing argument" for reasons that will become clear. 

\begin{lem}\label{L:ModuliRelZero2}
Suppose that $\cM$ is a CR subvariety and that $\bfC_1$ and $\bfC_2$ are distinct subequivalence classes of cylinders that happen to be parallel on $(X, \omega) \in \cM$. Then if one cylinder in $\bfC_1$ shares a boundary saddle connection with a cylinder in $\bfC_2$, every cylinder in $\bfC_1$ shares a boundary saddle connection with a cylinder in $\bfC_2$. 
\end{lem}
\begin{proof}
Without loss of generality suppose that the cylinders in $\bfC_1 \cup \bfC_2$ are horizontal. Suppose that $C_1$ and $C_2$ share a boundary saddle connection $s$ and suppose without loss of generality that $s$ lies on the bottom boundary of $C_2$ and where $C_i \in \bfC_i$. Shear $(X, \omega)$, i.e. apply an element of $\begin{pmatrix} 1 & t \\ 0 & 1 \end{pmatrix}$ for some $t \in \mathbb{R}$, to form a surface $(X', \omega')$ on which there are no vertical saddle connections and so that there is a singularity $z$ on the top boundary of $C_2$ that lies directly vertically above an interior point of $s$ (i.e. the vertical separatrix traveling down from $z$ into $C_2$ travels the height of $C_2$ down from $z$ before arriving at an interior point of $s$). Apply the standard dilation, i.e. travel in $\cM$ in the direction of $-i\sigma_{\bfC_2}$, until the heights of the cylinders in $\bfC_2$ reach zero. At this point, the singularity $z$ has accumulated on the saddle connection $s$. Continuing to travel in the direction of $-i\sigma_{\bfC_2}$ (we can travel in this direction indefinitely while remaining in $\cM$ since $(X', \omega')$ had no vertical saddle connections) the point $z$ moves into $C_1$ causing the height of $C_1$ to decrease. By Definition \ref{D:CS-AIS} \eqref{I:SC-Ratios}, the height of every cylinder in $\bfC_1$ must decrease and so every cylinder in $\bfC_1$ must have shared a boundary saddle connection with a cylinder in $\bfC_2$. 
\end{proof}

\begin{lem}\label{L:SECDecomp}
Let $U$ be an open subset of an invariant subvariety $\cM$. Suppose that $(X, \omega) \in U$ contains a collection of cylinders $\bfC$ that can be partitioned into subequivalence classes $\{ \bfC_i \}_{i=1}^n$ satisfying Definition \ref{D:CS-AIS} \eqref{I:SC-Persistence} and \eqref{I:SC-Ratios} for all deformations remaining in $U$. Then the following hold: 
\begin{enumerate}
    \item\label{I:SEC:LinComb} Any element of $\mathrm{Twist}(\bfC, \cM)$ can be expressed uniquely as a linear combination of $\{\sigma_{\bfC_i}\}_{i=1}^n$.
    \item\label{I:SEC:Support} The support of any element of $\mathrm{Twist}(\bfC, \cM)$ is a union of subequivalence classes of cylinders. 
    \item\label{I:SEC:SEC} Two cylinders in $\bfC$ are subequivalent if and only if there is some $i$ so that both cylinders belong to $\bfC_i$.
\end{enumerate}
\end{lem}
\begin{proof}
The second two claims follow immediately from the first, which we now prove. Suppose without loss of generality that the cylinders in $\bfC$ are horizontal. Since the real and imaginary parts of any element of $\mathrm{Twist}(\bfC, \cM)$ belong to $\mathrm{Twist}(\bfC, \cM)$ it suffices to prove the claim for purely imaginary elements. 

By Definition \ref{D:CS-AIS} \eqref{I:SC-Persistence} and \eqref{I:SC-Ratios}, the ratio of heights of any two subequivalent cylinders is locally constant in $U$. In particular, this implies that if $v \in \mathrm{Twist}(\bfC, \cM)$ is purely imaginary then it has the form 
\[ \sum_{i=1}^n a_i \sum_{C \in \bfC_i} h_C \gamma_C^* = \sum_{i=1}^n a_i \sigma_{\bfC_i} \]
where $a_i$ is purely imaginary and $h_C$ (resp. $\gamma_C^*$) denotes the height (resp. intersection pairing with the core curve) of $C$. The uniqueness is obvious. 
\end{proof}

% \begin{lem}\label{L:SECDecomp}
% Suppose that $\cM$ is a CS subvariety and that $\bfC$ is an equivalence class on a surface $(X, \omega) \in \cM$. Suppose that $\bfC_1, \hdots, \bfC_n$ are the subequivalence classes comprising $\bfC$. Then any element of $\mathrm{Twist}(\bfC, \cM)$ can be expressed uniquely as a linear combination of $\{\sigma_{\bfC_1}, \hdots, \sigma_{\bfC_n} \}$. In particular, the support of any element of $\mathrm{Twist}(\bfC, \cM)$ is a union of subequivalence classes of cylinders. 

% More generally, this result holds for any open subset of an invariant subvariety on which Definition \ref{D:CS-AIS} holds.
% \end{lem}
% \begin{proof}
% Suppose without loss of generality that the cylinders in $\bfC$ are horizontal. Since the real and imaginary parts of any element of $\mathrm{Twist}(\bfC, \cM)$ belong to $\mathrm{Twist}(\bfC, \cM)$ it suffices to prove the claim for purely imaginary elements. 

% By Definition \ref{D:CS-AIS} \eqref{I:SC-Persistence} and \eqref{I:SC-Ratios}, the ratio of heights of any two subequivalent cylinders is locally constant in $\cM$. In particular, this implies that if $v \in \mathrm{Twist}(\bfC, \cM)$ is purely imaginary then it has the form 
% \[ \sum_{i=1}^n a_i \sum_{C \in \bfC_i} h_C \gamma_C^* = \sum_{i=1}^n a_i \sigma_{\bfC_i} \]
% where $a_i$ is purely imaginary and $h_C$ (resp. $\gamma_C^*$) denotes the height (resp. intersection pairing with the core curve) of $C$. The uniqueness is obvious. 
% \end{proof}

\subsection{The boundary of a CR subvariety}

Say that two $\cM$-equivalent cylinders $C$ and $C'$ on a surface in $\cM$ are \emph{weakly-subequivalent} if there is a collection $\bfC$ of $\cM$-equivalent cylinders containing $C$ and $C'$ and so that $\mathrm{Twist}\left( \bfC, \cM \right)$ is one-dimensional with the support of any nonzero element being $\bfC$. It is clear that $C$ and $C'$ remain weakly subequivalent along any path on which the cylinders in $\bfC$ persist. We will say that $(X, \omega) \in \cM$ is a \emph{CR point} if the following hold:
\begin{enumerate}
    \item Two cylinders are weakly subequivalent if and only if they are subequivalent.
    \item Weakly subequivalent cylinders satisfy Definition \ref{D:CS-AIS} \eqref{I:SC-Ratios}.
    \item Definition \ref{D:CS-AIS} \eqref{I:SC-Partition} holds.
\end{enumerate}

\begin{lem}\label{L:PointwiseTest}
Let $U$ be an open subset of an invariant subvariety $\cM$. Definition \ref{D:CS-AIS} holds on $U$ if and only if every point in $U$ is a CR point. 
\end{lem}
\begin{proof}
Suppose first that Definition \ref{D:CS-AIS} holds on $U$. Then two cylinders are weakly subequivalent if and only if they are subequivalent by Lemma \ref{L:SECDecomp}. The other conditions are immediate.

Suppose now that every point in $U$ is a CR point. We must only verify Definition \ref{D:CS-AIS} \eqref{I:SC-Persistence}. It suffices to show that if $(X, \omega) \in U$ has two weakly subequivalent cylinders $C_1$ and $C_2$ that persist along a path $\gamma$ (contained in $U$) from $(X, \omega)$ to $(X', \omega')$, then $C_1$ and $C_2$ remain weakly subequivalent on $(X', \omega')$ 

Let $\bfC$ be the collection of cylinders on $(X, \omega)$ containing $C_1$ and $C_2$ and so that $\mathrm{Twist}\left( \bfC, \cM \right)$ is one-dimensional with the support of any nonzero element being $\bfC$. It suffices to show that the cylinders in $\bfC$ persist along $\gamma$. 

As long as $\bfC$ persists, the cylinders in it are pairwise weakly subequivalent and hence the ratio of moduli (resp. circumferences) of any two cylinders in $\bfC$ is contained in the finite set $S_1$ (resp. $S_2$). Since $C_1$ persists along $\gamma$ its modulus and circumference are bounded above and away from zero along $\gamma$. The same thus holds for the cylinders in $\bfC$ as long as they persist along $\gamma$. In other words, as long as the cylinders in $\bfC$ persist along $\gamma$ they have moduli and circumference bounded away from zero and bounded above. But this just means that the cylinders in $\bfC$ persist along $\gamma$ as desired. 
\end{proof}

\begin{lem}\label{L:Limits}
Let $U$ be an open subset of an invariant subvariety $\cM$. If every point in $U$ is a CR point, then the limit of any convergent sequence of points in $U$ is a CR point.
% Suppose that $U$ is a open subset of an invariant subvariety $\cM$ on which Definition \ref{D:CS-AIS} holds. Let $(X_n, \omega_n)$ be a sequence in $U$ that converges to a surface $(X, \omega)$. If $(X, \omega)$ belongs to a component of the boundary, call this component $\cM'$, otherwise, set $\cM' = \cM$. If $(X_n, \omega_n)$ are CS points, then so is $(X, \omega)$. 
% \begin{enumerate}
%     \item\label{I:limit:ec} If $\bfC$ is an equivalence class then $\sigma_{\bfC} \in T_{(X, \omega)} \cM'$.
%     \item\label{I:limit:wes} If two cylinders $C_1$ and $C_2$ are weakly subequivalent on $(X, \omega) \in \cM'$ then they are subequivalent and their ratio of moduli (resp. circumferences) belongs to $S_1$ (resp. $S_2$).
%     \item\label{I:limit:s} If two cylinders $C_1$ and $C_2$ are subequivalent on $(X, \omega) \in \cM'$ then they are weakly subequivalent. 
% \end{enumerate}
\end{lem}
\begin{proof}
Let $(X_n, \omega_n)$ be a sequence of points in $U$ converging to $(X, \omega)$. If $(X, \omega)$ belongs to a component of the boundary, call this component $\cM'$, otherwise, set $\cM' := \cM$.

The result will follow from the following three sublemmas. Each proof will begin in the same way. We will take a collection of cylinders $\bfC$ on $(X, \omega)$. By Theorem \ref{T:CylinderBoundaryTangentFormula} \eqref{I:boundary:FindingCn} and \eqref{I:boundary:TwistIsom}, after perhaps passing to a subsequence, we may find cylinders $\bfC_n$ on $(X_n, \omega_n)$ that correspond to those in $\bfC$ and so that this correspondence induces an isomorphism between $\mathrm{Twist}(\bfC, \cM')$ and $\mathrm{Twist}(\bfC_n, \cM)$. To avoid repeating this setup three times, we state it now once. But bear in mind that the cylinders in $\bfC$ will differ between the following three arguments.

\begin{sublem}
If $\bfC$ is an equivalence class then $\sigma_{\bfC} \in T_{(X, \omega)} \cM'$.
\end{sublem}
\begin{proof}
Let $\bfC_n'$ be the union of all cylinders on $(X_n, \omega_n)$ that are subequivalent to those in $\bfC_n$. Since the cylinders in $\bfC_n$ have heights and circumferences that converge the same must hold for the cylinders in $\bfC_n'$ after passing to a subsequence (by Definition \ref{D:CS-AIS} \eqref{I:SC-Ratios}). Therefore, after passing to a subsequence, we may assume that each cylinder in $\bfC_n'$ converges to one on $(X, \omega)$ (by Theorem \ref{T:CylinderBoundaryTangentFormula} \eqref{I:boundary:CnConvergence}). The limiting cylinder is necessarily equivalent to one in $\bfC$, so this shows that, for sufficiently large $n$, $\bfC_n = \bfC_n'$, i.e. that $\bfC_n$ is a union of subequivalence classes, and hence that $\sigma_{\bfC_n} \in T_{(X_n, \omega_n)} \cM$. By Theorem \ref{T:CylinderBoundaryTangentFormula} \eqref{I:boundary:TwistIsom}, $\sigma_{\bfC_n}$ corresponds to an element of $T_{(X, \omega)} \cM'$ and these elements converge to $\sigma_{\bfC}$. Hence, $\sigma_{\bfC} \in T_{(X, \omega)} \cM$.
\end{proof}

\begin{sublem}
If two cylinders $C_1$ and $C_2$ are weakly subequivalent on $(X, \omega) \in \cM'$ then they are subequivalent and their ratio of moduli (resp. circumferences) belongs to $S_1$ (resp. $S_2$).
\end{sublem}
\begin{proof}
Let $\bfC$ be the collection of cylinders on $(X, \omega)$ containing $C_1$ and $C_2$ and so that $\mathrm{Twist}\left( \bfC, \cM \right)$ is one-dimensional with the support of any nonzero element being $\bfC$. By Lemma \ref{L:SECDecomp} \eqref{I:SEC:Support}, $\bfC_n$ is a subequivalence class of cylinders whose ratios of moduli (resp. circumferences) belong to $S_1$ (resp. $S_2$). Since the cylinders in $\bfC_n$ converge to those in $\bfC$ the result follows as in the previous sublemma.
\end{proof}

\begin{sublem}
If two cylinders $C_1$ and $C_2$ are subequivalent on $(X, \omega) \in \cM'$ then they are weakly subequivalent. 
\end{sublem}
\begin{proof}
Let $\bfC$ be the collection of cylinders on $(X, \omega)$ containing $C_1$ and $C_2$ so that $\sigma_{\bfC} \in T_{(X, \omega)} \cM$. Since $\sigma_{\bfC}$ is an element of $\mathrm{Twist}(\bfC_n, \cM)$ whose support is $\bfC_n$, $\bfC_n$ is a union of subequivalence classes by Lemma \ref{L:SECDecomp} \eqref{I:SEC:Support}. After passing to a subsequence we may suppose that the partition of $\bfC_n$ into subequivalence classes does not depend on $n$ (recall that we have an identification of the cylinders in $\bfC_n$ with the cylinders in $\bfC$). Since the cylinders in $\bfC_n$ converge to $\bfC$, we see that $\bfC$ is partitioned into the same number of subequivalence classes as $\bfC_n$. Since $\bfC$ is itself a subequivalence class the same must hold for $\bfC_n$. This shows that $\mathrm{Twist}(\bfC_n, \cM)$ is one-dimensional and spanned by $\sigma_{\bfC_n}$ by Lemma \ref{L:SECDecomp} \eqref{I:SEC:LinComb}. Since the cylinders in $\bfC_n$ converge to $\bfC$, $\mathrm{Twist}(\bfC, \cM)$ is one-dimensional and spanned by $\sigma_{\bfC}$, which shows that any two cylinders in $\bfC$ are weakly subequivalent.
\end{proof}
\end{proof}

\begin{cor}\label{C:PrimeTest}
Suppose that $U$ is an open subset of a prime invariant subvariety $\cM$. If each point in $U$ is a CR point then $\cM$ is a CR subvariety.
\end{cor}
\begin{proof}
Since $\cM$ is prime, the $\mathrm{GL}(2, \mathbb{R})$ action is ergodic on $\cM$ (by Chen-Wright \cite[Corollary 7.4]{ChenWright}). This implies that $U' := \bigcup_{g \in \mathrm{GL}(2, \mathbb{R})} gU$ is open, dense, and that every point in it is a CR point. Since $U$ is open and dense in $\cM$, every point in $\cM$ is a CR point (by Lemma \ref{L:Limits}). Therefore, $\cM$ is a CR subvariety by Lemma \ref{L:PointwiseTest}.
\end{proof}

\begin{prop}\label{P:Rank1.5}
Suppose that $\cM$ is a CR subvariety whose elements are connected surfaces. If $\cM'$ is a component of the boundary of $\cM$, then it is a CR subvariety. 
\end{prop}

The requirement that the surfaces in $\cM$ are connected is present solely because it is required to apply the boundary theory developed in Mirzakhani-Wright \cite{MirWri} and Chen-Wright \cite{ChenWright}.

\begin{proof}
By Chen-Wright \cite[Proposition 2.5]{ChenWright} (see also Mirzakhani-Wright \cite[Proposition 2.6]{MirWri}), there is an open subset $U \subseteq \cM'$ consisting of limits of sequences of surfaces in $\cM$. By Lemma \ref{L:Limits}, every point in $U$ is a CR point. Let $\cM' = \cM_1 \times \hdots \times \cM_k$ be the prime decomposition of $\cM'$ (see Definition \ref{D:Prime}). Since $\cM_i$ is prime and since every point in the projection of $U$ to $\cM_i$ is a CR point, $\cM_i$ is a CR subvariety by Corollary \ref{C:PrimeTest}. Since no cylinders on a surface in $\cM_i$ can be equivalent or subequivalent to cylinders on a surface in $\cM_j$ if $i \ne j$, it follows that $\cM$ itself is a CR subvariety.
\end{proof}

\section{The base case of the proof of Theorem \ref{T:main}}

This section is dedicated to the proof of the following.

\begin{prop}\label{P:BaseBase}
If $\cM$ is a rank two rel zero invariant subvariety of minimal homological dimension, then $\cM$ is geminal.
\end{prop}

The general definition of ``geminal" appears in Definition \ref{D:Geminal}, but when $\cM$ has rel zero, it suffices to show that every equivalence class of cylinders consists of isometric cylinders and has at most two elements. To this end, suppose that $\bfC_1$ is an equivalence class of cylinders on a surface $(X, \omega)$ in $\cM$. 

Without loss of generality suppose that the cylinders in $\bfC_1$ are horizontal. By perturbing, we may suppose moreover that $(X, \omega)$ is cylindrically stable (by Theorem \ref{L:Stable} \eqref{I:PerturbToStableQ} and Corollary \ref{C:FieldDefinition}). We will perturb so that if two cylinders in $\bfC_1$ are not isometric before perturbing they are not isometric after perturbing. Note that after perturbing $\bfC_1$ remains an equivalence class, and not just a subset of one (by Lemmas \ref{L:ModuliRelZero} and \ref{L:NoNewECsInRelZero}).

By Lemma \ref{L:Stable} \eqref{I:MaxEC} there is exactly one other equivalence class $\bfC_2$ of horizontal cylinders. In particular, there is a cylinder in $\bfC_1$ that shares a boundary saddle connection with a cylinder in $\bfC_2$.

\begin{lem}\label{L:TwoCylinders}
There are at most two cylinders in $\bfC_1$. If there are two, then the top boundary of one coincides with the bottom boundary of the other.
\end{lem}
\begin{proof}
Let $n$ be the number of cylinders in $\bfC_1$. We will show that $n = 2$. Since $\cM$ has minimal homological dimension, fixing $i \in \{1, 2\}$, any two cylinders in $\bfC_i$ have core curves that are homologous to each other (by Theorem \ref{T:MHDImpliesHomology}). Cutting the core curves of the cylinders in $\bfC_1$ creates $n-1$ subsurfaces, each of which has two boundary components - each one corresponding to a core curve of a cylinder in $\bfC_1$. Since the core curves of cylinders in $\bfC_2$ are all homologous to each other (and not homologous to the core curve of any cylinder in $\bfC_1$), the cylinders in $\bfC_2$ belong to exactly one of these subsurfaces and hence there are at most two cylinders in $\bfC_1$ which can share a boundary saddle connection with a cylinder in $\bfC_2$. By Lemma \ref{L:ModuliRelZero2}, every cylinder in $\bfC_1$ shares a boundary saddle connection with a cylinder in $\bfC_2$, so there are at most two cylinders in $\bfC_1$. 

Since the cylinders in $\bfC_2$ are contained in exactly one subsurface of $(X, \omega) - \overline{\bfC_1}$, it follows that if there are two cylinders in $\bfC_1$ the top boundary of one must coincide with the bottom boundary of the other.
\end{proof}

Since there is nothing more to prove when $\bfC_1$ consists of a single cylinder, we suppose that $\bfC_1$ consists of two cylinders, which we call $C$ and $C'$.

To conclude the proof of Proposition \ref{P:BaseBase}, it suffices to show that the cylinders in $\bfC_1$ are isometric. Since $\cM$ has minimal homological dimension, the core curves of cylinders in $\bfC_1$ are homologous (by Theorem \ref{T:MHDImpliesHomology}) and hence have equal length. So it suffices to show that the cylinders in $\bfC_1$ have equal heights. 

Since the cylinders in $\bfC_2$ are not generically parallel to those in $\bfC_1$, there is a surface $(X', \omega')$ that is arbitrarily close to $(X, \omega)$ on which the cylinders in $\bfC_1$ remain horizontal, the cylinders in $\bfC_2$ are not, and the cylinders in $\bfC_2$ are generic. We will assume that this perturbation was performed so that if $C$ and $C'$ had unequal height before the perturbation they did so after as well.  As before, $\bfC_1$ and $\bfC_2$ remain equivalence classes, and not just subsets of one (by Lemmas \ref{L:ModuliRelZero} and \ref{L:NoNewECsInRelZero}). By applying the matrix $\begin{pmatrix} 1 & t \\ 0 & 1 \end{pmatrix}$ we will suppose that $\bfC_2$ is vertical on $(X', \omega')$. After applying the standard shear in $\bfC_2$ (which is possible by the cylinder deformation theorem) we may also suppose that, on $(X', \omega')$, there are cylinders in $\bfC_2$ that contain horizontal saddle connections. By \cite[Lemma 9.1]{Apisa}, $(X', \omega')$ is horizontally periodic and not covered by the cylinders in $\bfC_1$ since these cylinders do not intersect $\bfC_2$. Let $\bfC_3$ be the equivalence class of the horizontal cylinders that intersect $\bfC_2$. Since equivalent cylinders are homologous (by Theorem \ref{T:MHDImpliesHomology}) and since the cylinders in $\bfC_1$ do not intersect those in $\bfC_2$, we have that $\bfC_3 \ne \bfC_1$. $\bfC_1$ and $\bfC_3$ are the only two equivalence classes of horizontal cylinders on $(X', \omega')$ (by Lemma \ref{L:Stable} \eqref{I:MaxEC}).

\begin{lem}\label{L:BaseCase:BoundaryStructure}
$\ColTwo(X', \omega')$ is connected and the core curves of cylinders in $\ColTwo(\bfC_1 \cup \bfC_3)$ are all homologous. Moreover, $\ColTwo(\bfC_1)$ and $\ColTwo(\bfC_3)$ are subequivalence classes of cylinders. 
\end{lem}
\begin{proof}
The first claim follow from Corollary \ref{C:MHD-Boundary:Connected}, which also implies that $\cM_{\bfC_2}$ has minimal homological dimension. Since $\cM$ has rank two rel zero, $\cM_{\bfC_2}$ must have rank one and hence if two cylinders are parallel on a surface in $\cM_{\bfC_2}$ they are $\cM_{\bfC_2}$-equivalent. Since the cylinders in $\ColTwo(\bfC_1)$ and $\ColTwo(\bfC_3)$ are parallel they must be $\cM_{\bfC_2}$-equivalent and hence the second claim follows by Theorem \ref{T:MHDImpliesHomology}. The third claim follows from Theorem \ref{T:CylinderBoundaryTangentFormula} \eqref{I:boundary:TwistIsom}.
\end{proof}

Our strategy will be to overcollapse the cylinders in $\Col_{\bfC_2}(\bfC_3)$ as in the proof of Lemma \ref{L:ModuliRelZero2} and to conclude that the heights of each cylinder in $\Col_{\bfC_2}(\bfC_1)$ change at the same rate under this deformation, which will imply that they have equal height.

\begin{rem}
The results in Lemma \ref{L:BaseCase:BoundaryStructure} are the reason we pass to the boundary. If we were to attempt an overcollapsing argument without degenerating it would be hard to control how the heights of the cylinders in $\bfC_1$ change. However, the extra information that all four cylinders in $\Col_{\bfC_2}\left( \bfC_1 \cup \bfC_3 \right)$ are homologous makes it easy to compute the change in heights of the cylinders in $\Col_{\bfC_2}(\bfC_1)$ after overcollapsing.
\end{rem}

Begin by shearing $\ColTwo(X', \omega')$ so that there are no vertical saddle connections. Let $(Y, \eta)$ denote the resulting surface. 

To fix notation, suppose that $D$ is a cylinder of height $h$ in $\Col_{\bfC_2}(\bfC_3)$ on $(Y, \eta)$. If there is a second cylinder in $\Col_{\bfC_2}(\bfC_3)$ we call it $D'$ and denote its height by $h'$. Now, as in the proof of Lemma \ref{L:ModuliRelZero2}, we will overcollapse $\Col_{\bfC_2}(\bfC_3)$ into $\Col_{\bfC_2}(\bfC_1)$, that is, we will flow from $(Y, \eta)$ in $\cM$ in the direction of $-i\sigma_{\ColTwo(\bfC_3)}$ until the heights of the cylinders in $\Col_{\bfC_2}(\bfC_3)$ are zero. Then we will continue flowing in the direction of $-i\sigma_{\Col_{\bfC_2}(\bfC_3)}$. 

To be concrete, we will let $(Y_t, \eta_t) := (Y, \eta) - (1-t)i \sigma_{\ColTwo(\bfC_3)}$. Since there are no vertical saddle connections on $(Y, \eta)$ this surface is defined for all real $t$. 

\begin{lem}\label{L:SameHeight1}
The derivatives, with respect to $t$, of the heights of the cylinders of $\ColTwo(\bfC_1)$ on $(Y_t, \eta_t)$ for $t \in (-\epsilon, 0)$ are equal for any sufficiently small $\epsilon$. 
\end{lem}
\begin{proof}
Suppose without loss of generality that $\Col_{\bfC_2}(C)$ borders $D$ along its top boundary. On $(Y_0, \eta_0)$ new singularities have accumulated on the top boundary of $\Col_{\bfC_2}(C)$. Intuitively, these singularities ``came from $\Col_{\bfC_2}(\bfC_3)$". However, since the bottom boundary of $\Col_{\bfC_2}(C)$ coincides with the top boundary of $\Col_{\bfC_2}(C')$, no new singularities have accumulated there. 

Triangulate $\Col_{\bfC_2}(C)$ by saddle connections on $(Y_0, \eta_0)$ using all the saddle connections on the boundary of $\Col_{\bfC_2}(C)$ as edges of the triangulation and then adding in additional saddle connections, the set of which we call $S$, that join one boundary of $\Col_{\bfC_2}(C)$ to the other. Extend the triangulation to a triangulation $T$ by saddle connections of $(Y_0, \eta_0)$. There is an open subset $U$ of the stratum containing $(Y_0, \eta_0)$ on which the cylinders in $\ColTwo(\bfC_1)$ persist and so that $T$ remains a triangulation by saddle connections for every surface in $U$. Select $\epsilon$ so that $(Y_t, \eta_t)$ belongs to $U$ for all $t \in (-\epsilon, +\epsilon)$. In particular, for these values of $t$, the saddle connections in $S$ remain saddle connections and the height of $\Col_{\bfC_2}(C)$ on $(Y_t, \eta_t)$ for $t \in (-\epsilon, \epsilon)$ is given by $\min_{s \in S}\left( \eta_t(s) \right)$ (this uses the fact that, along this path, no new singularities accumulate on the bottom boundary of $\ColTwo(C)$). 

\begin{sublem}
For $t \in (0, \epsilon)$, each saddle connection in $S$ connects a point on the bottom boundary of $\Col_{\bfC_2}(C)$ to one of the following: a point in the top boundary of $\Col_{\bfC_2}(C)$, a point in the top boundary of $D$, and (provided that $D'$ exists) a point in the top boundary of $D'$. The derivatives of the imaginary parts of the periods of saddle connections in $S$ with respect to $t$ are, respectively,  $0$, $-h$, and $-(h+h')$.
\end{sublem}
\begin{proof}
The trichotomy is obvious. Since periods vary linearly with $t$ along the path, it suffices to show that, on $(Y_t, \eta_t)$ for $t \in (0, \epsilon)$, any saddle connection in $S$ crosses each horizontal cylinder at most once. 

Suppose to a contradiction that a saddle connection $s \in S$ crosses a horizontal cylinder on $(Y, \eta)$ more than once. By adding part of the core curve of that cylinder to $s$ it is possible to form a simple closed curve that has positive intersection number with some horizontal cylinders, but not with $\ColTwo(C')$, into which $s$ does not cross. This contradicts the fact that the four horizontal cylinders on $(Y_t, \eta_t)$, for $t \in (0, \epsilon)$, have homologous core curves (by Lemma \ref{L:BaseCase:BoundaryStructure}). 
\end{proof}

When $\Col_{\bfC_2}(\bfC_3) = \{D\}$ for $t \in (-\epsilon, 0)$ the derivatives of the heights of $\ColTwo(C)$ and $\ColTwo(C')$ are both equal to $-h$. This follows since, because new singularities have accumulated on the top boundary of $\ColTwo(C)$, one of the saddle connections in $S$ joins a zero on the bottom boundary of $\ColTwo(C)$ to a zero on the top boundary of $D$ on $(Y_t, \eta_t)$ for $t \in (0, \epsilon)$. Similarly, when $\Col_{\bfC_2}(\bfC_3) = \{D, D'\}$, by construction, there is a singularity from the top boundary of $D'$ that accumulates on the boundary of $\ColTwo(C)$. As before, this shows that the derivative of the height of $\ColTwo(C)$ for $t \in (-\epsilon, 0)$ is $-(h+h')$. By symmetry of hypotheses the same holds for $\ColTwo(C')$.
\end{proof}

By Proposition \ref{P:Rank1.5}, since the cylinders in $\ColTwo(\bfC_1)$ are subequivalent on $(Y, \eta)$ and persist on $(Y_t, \eta_t)$ for all $t \in (-\epsilon, 1]$, it follows that the ratio of moduli of cylinders in $\ColTwo(\bfC_1)$ is constant for all $t$. By Lemma \ref{L:SameHeight1} this ratio must be one. Since the ratio of heights of cylinders in $\ColTwo(\bfC_1)$ is the same as the ratio of moduli (since the cylinders in $\ColTwo(\bfC_1)$ have homologous core curves) it follows that heights of the cylinders in $\ColTwo(\bfC_1)$ are equal on $(Y, \eta)$. Since the heights of the cylinders in $\ColTwo(\bfC_1)$ on $(Y, \eta)$ are the same as the heights of the cylinders in $\bfC_1$, the cylinders in $\bfC_1$ have identical heights, as desired. 

%%%%%%%%%%%%%%%%%%%%%%%
%
%
% SECTION - GEMINAL PRELIMINARIES
%
%
%%%%%%%%%%%%%%%%%%%%%%%

\section{Preliminaries on geminal varieties}\label{S:GeminalandMP}

In this section we will assemble various results and establish simple consequences of them that we will require for the proof of Theorem \ref{T:main}.

%%%%%%%%%%%%%%%%%%%%
\subsection{Minimal covers}

One of the most important objects in the sequel will be the $\emph{minimal cover}$ associated to a translation surface. This cover is defined by the following result, which is a combination of a result of M\"oller \cite[Theorem 2.6]{M2} and an extension found in \cite[Lemma 3.3]{ApisaWright}.   

\begin{thm}\label{T:MinimalCover}
Suppose that $(X, \omega)$ is not a torus cover. There is a unique translation surface $(X_{min}, \omega_{min})$ and a translation covering $$\pi_{X_{min}}: (X, \omega) \rightarrow (X_{min}, \omega_{min})$$ such that any translation cover from $(X, \omega)$ to a translation surface is a factor of $\pi_{X_{min}}$. 

Additionally, there is a quadratic differential $(Q_{min}, q_{min})$ with a degree
1 or 2 map $(X_{min}, \omega_{min}) \rightarrow (Q_{min}, q_{min})$ such that any map from $(X, \omega)$ to a quadratic differential is a factor of the composite map $\pi_{Q_{min}}:(X, \omega) \rightarrow (Q_{min}, q_{min})$.
\end{thm}

%%%%%%%%%%%%%%%%%%%%
\subsection{Marked points}

We begin with the following two definitions. 

\begin{defn}\label{D:Forgetting}
Given a translation surface $(X, \omega)$ with marked points let $\For(X, \omega)$ be the surface with the marked points forgotten. If $\cM$ is the orbit closure of $(X, \omega)$, then $\For(\cM)$ will denote the orbit closure of $\For(X, \omega)$. 
\end{defn}

\begin{defn}\label{D:GenericPoints}
Let $(X, \omega)$ be a surface with dense orbit in an invariant subvariety $\cM$. Given a finite collection of points $P$ on $(X, \omega)$ we will let $(X, \omega; P)$ denote the result of marking the points $P$ on $(X, \omega)$. The points $P$ are said to be \emph{generic} if the following occurs
\[ \dim \overline{\mathrm{GL}(2, \mathbb{R}) \cdot (X, \omega; P)} = |P| + \dim \cM. \]
A point $p$ on $(X, \omega)$ is said to be \emph{periodic} if $p$ is not a zero of $\omega$ and 
\[ \dim \overline{\mathrm{GL}(2, \mathbb{R}) \cdot (X, \omega; \{p\})} = \dim \cM. \]
Finally, if $Q$ denotes the collection of marked points on $(X, \omega)$ then a point $q \in Q$ is said to be \emph{free} if it can be moved around the surface while remaining in $\cM$ and while fixing the position of all other marked points and the underlying unmarked surface.
\end{defn}

\begin{thm}[Apisa \cite{Apisa}]\label{T:PeriodicPoints-NonHyp}
If $(X, \omega)$ is a translation surface of genus at least two so that $\For(X, \omega)$ has dense orbit in a non-hyperelliptic component of a stratum, then any collection of points on $(X, \omega)$ that does not include zeros of $\omega$ is generic. 
%
% Given a translation surface $(X, \omega)$ so that $\For(X, \omega)$ has dense orbit in a non-hyperelliptic component of a stratum of Abelian differentials of genus at least two, any collection of points $(X, \omega)$ that does not include zeros of $\omega$ is generic. 
\end{thm}

The following result is an immediate consequence of Apisa-Wright \cite{ApisaWright}.

\begin{thm}\label{T:PeriodicPointsHyp}
Let $(X, \omega)$ be a translation surface of genus at least two with marked points $Q$ and so that $\For(X, \omega)$ has dense orbit in a hyperelliptic locus. Let $\cM$ be the orbit closure of $(X, \omega; Q)$. The marked points on a surface in $\cM$ consist of free points, pairs of points exchanged by the hyperelliptic involution, fixed points of the hyperelliptic involution, and zeros of $\omega$, with no further constraints on the points.
\end{thm}
\begin{proof}
Every hyperelliptic locus is a locus of holonomy double covers of a stratum $\cQ$ of genus zero quadratic differentials. The rank of $\cQ$ is the same as the genus of $X$, so in our case, the rank of $\cQ$ is at least two. By Apisa-Wright \cite[Theorem 1.4]{ApisaWright}, there are no periodic points on a surface with dense orbit in $\cQ$; this uses that $\cQ$ has rank at least two and the fact that a stratum of genus zero quadratic differentials never contains a hyperelliptic connected component. By Apisa-Wright \cite[Lemma 4.5]{ApisaWright}, the quotient of $(X, \omega)$ by the hyperelliptic involution is $\pi_{Q_{min}}$. The lemma now follows from Apisa-Wright \cite[Theorem 1.3]{ApisaWright}.
\end{proof}

\begin{cor}\label{C:PeriodicPointsHyp}
If $\cM$ is a hyperelliptic locus of translation surfaces of genus at least two, then any generic cylinder on any surface in $\cM$ is simple or half-simple. 
\end{cor}
\begin{proof}
Suppose that $p$ is a marked point that is neither a zero nor a fixed point of the hyperelliptic involution. By Theorem \ref{T:PeriodicPointsHyp}, if $p$ appears on the boundary of a generic cylinder $C$ then there is a saddle connection joining $p$ to itself that comprises one boundary of $C$ (if not then it would be possible to move $p$ so the saddle connection with one endpoint at $p$ in the boundary of $C$ ceases to be parallel to the core curve of $C$). It therefore suffices to show that if $(X, \omega)$ only has fixed points of the hyperelliptic involution marked then the generic cylinders are simple or half-simple. In this case, $\cM$ is a quadratic double and the result follows from Apisa-Wright \cite[Lemma 12.4]{ApisaWrightHighRank}.
\end{proof}

\subsection{Geminal varieties}

In this subsection we will present a class of invariant subvarieties that were originally studied in Apisa-Wright \cite{ApisaWrightGeminal}. 

\begin{defn}\label{D:Geminal}
An invariant subvariety $\cM$ is said to be \emph{geminal} if for any cylinder $C$ on any $(X,\omega)\in \cM$, either 
\begin{itemize}
\item any cylinder deformation of $C$ remains in $\cM$, or 
\item there is a cylinder $C'$ such that $C$ and $C'$ are parallel and have the same height and circumference on $(X,\omega)$ as well as on all small deformations of $(X,\omega)$ in $\cM$, and any cylinder deformation that deforms $C$ and $C'$ equally remains in $\cM$.  
\end{itemize}
In the first case we say that $C$ is \emph{free}, and in the second case we say that $C$ and $C'$ are \emph{twins}. $\cM$ is called \emph{$h$-geminal} if additionally twin cylinders have homologous core curves. The partition of an equivalence class of cylinders into free cylinders and pairs of twins is called a \emph{geminal partition} 
% An invariant subvariety $\cM$ is called \emph{geminal} if for any equivalence class $\bfC$ on any surface $(X, \omega)$ in $\cM$, it is possible to partition $\bfC$ into subequivalence classes, each of which have at most two cylinders, all of which are isometric to each other. $\cM$ is called \emph{$h$-geminal} if additionally subequivalent cylinders have homologous core curves. 
\end{defn}

The simplest example of a geminal invariant subvariety, other than a component of a stratum, is a \emph{quadratic double}, i.e. a full locus of holonomy double covers where the collection of marked points is fixed by the holonomy involution and otherwise unconstrained. Before proceeding we establish the following useful corollary of the results of the previous subsection. 

\begin{cor}\label{C:HypParallelism}
Let $\bfC$ be an equivalence class of generic cylinders on a surface $(X, \omega)$ in a hyperelliptic locus $\cM$. If $\bfC$ does not admit a geminal partition, then there is a cylinder in $\bfC$ with a free marked point on its boundary.
\end{cor}
\begin{proof}
Without loss of generality, suppose that the only marked points on $(X, \omega)$ lie on the boundary of cylinders in $\bfC$. If all the marked points were fixed by the hyperelliptic involution, then $\cM$ would be a quadratic double and hence geminal. Therefore, not all the marked points are fixed by the hyperelliptic involution and so one of them must be free by Theorem \ref{T:PeriodicPointsHyp}
\end{proof}

To state the strongest possible result on geminal subvarieties we recall the following definitions.

\begin{defn}\label{D:GoodAndOptimal}
Let $(X, \omega)$ be a translation surface. A translation cover $f: (X, \omega) \rightarrow (X', \omega')$ will be called \emph{good} if every cylinder $C$ on $(X, \omega)$ is the preimage of its image under $f$. The cover will be called \emph{optimal} if it is good and any other good map is a factor of it.

If additionally $(X, \omega)$ belongs to an invariant subvariety $\cM$, we will say that $f$ is \emph{$\cM$-generic} if the map $f$ can be deformed to every nearby surface in $\cM$. A map $f$ defined on $(X,\omega)$ will be called \emph{$\cM$-good} if it is $\cM$-generic and the deformations of the map are good on all deformations of $(X, \omega)$ in $\cM$. The map $f$ will be called \emph{$\cM$-optimal} if it is $\cM$-good and any other $\cM$-good map is a factor of it. 

Finally, if $(X, \omega)$ is a torus cover we will let $\pi_{abs}$ denote the minimal degree map to the torus.
\end{defn} 

\begin{rem}\label{R:WhyWeCanPerturb}
If $(X, \omega)$ has dense orbit in an invariant subvariety of rank at least two, then $\pi_{X_{min}}$ is $\cM$-generic. Given a path in $\cM$ from $(X, \omega)$ to another surface $(X', \omega')$, $\pi_{X_{min}}$ must become a translation cover on $(X', \omega')$. This cover is exactly $\pi_{X'_{min}}$ when $(X', \omega')$ has dense orbit. This observation will allow us to make statements about $\pi_{X_{min}}$ and then prove them even after perturbing to a nearby surface. 
\end{rem}

The following result is Apisa-Wright \cite[Lemma 6.14]{ApisaWrightGeminal}. Note that the statement of Theorem \ref{T:geminal1} \eqref{I:R1HGeminal4} is slightly stronger than what is stated there, but that it follows from the proof.

%of Apisa-Wright \cite[Lemma 6.14]{ApisaWrightGeminal}, but is slightly stronger than the statement. 

\begin{thm}\label{T:geminal1}
Suppose that the orbit closure of $(X, \omega)$ is a rank one subvariety $\cM$ of minimal homological dimension.
\begin{enumerate}
\item\label{I:R1HGeminal1} The minimal degree map $\pi_{abs}$ from $(X, \omega)$ to a torus is the optimal map.
\item\label{I:R1HGeminal4} For any collection of cylinders $\bfC$ on $(X, \omega)$ so that $\sigma_{\bfC} \in T_{(X, \omega)} \cM$, $\Col_{\bfC}(\pi_{abs})$ is the optimal map for the connected surface $\Col_{\bfC}(X, \omega)$. 
\end{enumerate} 
\end{thm}

% A similar result (Theorem \ref{T:geminal2}) holds for $h$-geminal subvarieties, but before stating it we will require a few definitions and results. 
%
%The following is almost entirely contained in Apisa-Wright \cite[Proposition 8.1]{ApisaWrightGeminal}. 

 The following result is mostly contained in Apisa-Wright \cite[Proposition 8.1]{ApisaWrightGeminal}.

\begin{thm}\label{T:geminal2} 
Suppose that $\cM$ is an $h$-geminal invariant subvariety of rank at least two.  
\begin{enumerate}
\item\label{I:geminal2:PiOpt} Every surface  in $\cM$ has an $\cM$-optimal map $\pi_{opt}$, which coincides with $\pi_{X_{min}}$ for any surface with dense orbit in $\cM$.
\item\label{I:geminal2:WhatIsM} $\cM$ is either a stratum of Abelian differentials or a full locus of covers of a quadratic double of a genus zero stratum.
\item\label{I:geminal2:TwoCylinders} If the degree of the optimal map is greater than one, then every cylinder has a twin.
\item\label{I:geminal2:PiOptDegeneratesToPiOpt} If $\bfC$ is a collection of generic equivalent cylinders on a surface $(X, \omega)$ in  $\cM$ such that $\sigma_{\bfC} \in T_{(X, \omega)} \cM$ and so that $\cM_{\bfC}$ has dimension exactly one less than that of $\cM$, then $\Col_{\bfC}(\pi_{opt})$ is the $\cM_{\bfC}$-optimal map for $\Col_{\bfC}(X, \omega)$.
\end{enumerate}
\end{thm}
\begin{proof}
Noting that the only Abelian or quadratic doubles that are $h$-geminal are quadratic doubles of genus zero strata, the only item that is not immediately implied by Apisa-Wright \cite[Proposition 8.1]{ApisaWrightGeminal} is \eqref{I:geminal2:PiOptDegeneratesToPiOpt}. Suppose without loss of generality that $\bfC$ is horizontal.

We begin by noting that, since $\sigma_{\bfC} \in T_{(X, \omega)} \cM$, $\bfC$ can be partitioned into free cylinders and twins, i.e. subequivalence classes, which we denote by $\bfC_1, \hdots, \bfC_n$. Moreover, 
\[ \Col_{\bfC}(X, \omega) = \Col_{\bfC_{\sigma(1)}} \hdots \Col_{\bfC_{\sigma(n)}}(X, \omega) \]
for any permutation $\sigma \in \mathrm{Sym}(n)$. Since, by assumption, $\cM_{\bfC}$ has codimension one, there is a unique subequivalence class, which is $\bfC_1$ up to re-indexing, that contains a vertical saddle connection.  By Apisa-Wright \cite[Proposition 8.1]{ApisaWrightGeminal}, $\Col_{\bfC_1}(\pi_{opt})$ is the $\cM_{\bfC_1}$-optimal map. Since none of the subequivalence classes $\bfC_2, \hdots, \bfC_n$ contain vertical saddle connections, applying $\Col_{\bfC_i}$ for $i > 1$ does not pass to the boundary of $\cM_{\bfC_1}$, so $\Col_{\bfC}(\pi_{opt})$ is the $\cM_{\bfC}$-optimal map for $\Col_{\bfC}(X, \omega)$. 
%
% It remains to show that $\Col_{\bfC}(X, \omega)$ is connected. Using the same notation as before, it suffices to show that $\Col_{\bfC_1}(X, \omega)$ is connected. 
%
% When the degree of $\pi_{opt}$ is greater than one, this statement is an immediate application of \cite[Lemmas 4.24 and 4.25]{ApisaWrightGeminal}. So suppose the degree of $\pi_{opt}$ is one. In particular, $\cM$ is either a stratum of Abelian differentials or a quadratic double of a genus zero stratum by \eqref{I:geminal2:WhatIsM}. When $\cM$ is a stratum, the subequivalence classes of generic cylinders are singletons containing simple cylinders, so the connectedness of $\Col_{\bfC}(X, \omega)$ is immediate. When $\cM$ is a genus zero quadratic double, the subequivalence classes of generic cylinders consist of simple and half-simple cylinders (by Corollary \ref{C:PeriodicPointsHyp}) and so the connectedness of $\Col_{\bfC}(X, \omega)$ is again immediate.
\end{proof}

Recall that if $(X, \omega)$ is hyperelliptic then a \emph{Weierstrass point} is a fixed point of the hyperelliptic involution. A Weierstrass point is said to be \emph{regular} if it is not a zero of $\omega$. 

\begin{cor}\label{C:h-geminal-branching}
If $\cM$ is $h$-geminal, rank at least two, and the $\cM$-optimal map $\pi_{opt}$ is not the identity, then all but perhaps one of the regular Weierstrass points are branch points of $\pi_{opt}$.
\end{cor}
\begin{proof}
Let $(X, \omega) \in \cM$. If $\pi_{opt}$ is not the identity, then every cylinder on $\pi_{opt}(X, \omega)$ has a twin (by Theorem \ref{T:geminal2} \eqref{I:geminal2:TwoCylinders}) and hence $\pi_{opt}(X, \omega)$ belongs to a hyperelliptic locus (by Theorem \ref{T:geminal2} \eqref{I:geminal2:WhatIsM}) which we will denote by $\cM_{opt}$.

\begin{sublem}\label{SL:MarkedWP}
All but at most one regular Weierstrass point is a marked point. 
\end{sublem}
\begin{proof}
Every cylinder on $\pi_{opt}(X, \omega)$ has a twin by Theorem \ref{T:geminal2} \eqref{I:geminal2:TwoCylinders}. If two regular Weierstrass point were unmarked then $\cM_{opt}$ contains a surface where these two Weierstrass points lie on the core curve of a cylinder $C$ (for instance since there is some cylinder between two regular Weierstrass points and there is a path in $\cM_{opt}$ moving any two regular Weierstrass points to any two others while fixing the underlying surface. This follows for instance since it is possible to arbitrarily permute the poles on a genus zero quadratic differential while fixing all zeros). Since $C$ is fixed by the hyperelliptic involution it is a free cylinder, contradicting Theorem \ref{T:geminal2} \eqref{I:geminal2:TwoCylinders}.
\end{proof}

\begin{sublem}\label{SL:ForgettingWPs}
Let $p$ be a regular Weierstrass point on $\pi_{opt}(X, \omega)$ that is not a branch point. Let $\For'(X, \omega)$ denote $(X, \omega)$ with the marked points in $\pi_{opt}^{-1}(p)$ forgotten. Let $\cM'$ denote the orbit closure of $\For'(X, \omega)$. Then $\cM'$ remains geminal. 
\end{sublem}
\begin{proof}
On $\pi_{opt}(X, \omega)$ any generic cylinder that contains $p$ in its boundary has a twin, i.e. the image of the cylinder under the hyperelliptic involution is distinct. This shows that no free generic cylinder on $(X, \omega)$ contains a point in $\pi_{opt}^{-1}(p)$ in its boundary. Therefore, free cylinders on $(X, \omega)$ remain free cylinders after forgetting these points. 

Therefore if a cylinder $D$ in $(X, \omega)$ contains a point of $\pi_{opt}^{-1}(p)$ in its boundary, it has a twin $D'$. Moreover, there is a cylinder $C$ on $\pi_{opt}(X, \omega)$ that is divided into two by the marked Weierstrass point $p$ and another, not necessarily marked, Weierstrass point $p'$, and so $D \cup D' = \pi_{opt}^{-1}(C)$. If $p'$ is marked, then $D$ and $D'$ remain cylinders on $\For'(X, \omega)$ and in particular they remain twins. If $p'$ is unmarked, then $D \cup D'$ becomes a single free cylinder on $\For'(X, \omega)$. This shows that $\cM$ remains geminal. 
\end{proof}

If there are two regular Weierstrass points on $\pi_{opt}(X, \omega)$ that are not branch points of $\pi_{opt}$, then forgetting them (as in Sublemma \ref{SL:ForgettingWPs}), produces a geminal invariant subvariety that contradicts Sublemma \ref{SL:MarkedWP}.
%
% If $\pi_{opt}$ were not branched over two regular Weierstrass points, then forgetting these two regular Weierstrass points and their preimages yields a geminal invariant subvariety, contradicting the first sublemma. 
\end{proof}

%%%%%%%%%%%%%%
%
% SECTION - h-geminal up to marked points
%
%%%%%%%%%%%%%%

\section{$h$-geminal subvarieties up to marked points}

In this section we will study invariant subvarieties that are $h$-geminal up to marked points. We will show that every surface in such a variety has an optimal map and study limits of that map under general cylinder degenerations. 

\begin{defn}
An invariant subvariety $\cM$ is said to be \emph{$h$-geminal up to marked points} if there is a surface $(X, \omega)$ in $\cM$ with dense orbit and a finite collection $P$ of $\cM$-periodic points on $(X, \omega)$ so that the orbit closure of $(X, \omega; P)$ is $h$-geminal. 

If $\cM$ has rank at least two, then let $\cM^{pt}$ denote $\cM$ with all its periodic points marked. Since $\cM$ has rank at least two, only a finite number of periodic points must be added by Eskin-Filip-Wright \cite[Theorem 1.5]{EFW} (see Apisa-Wright \cite[Section 4.2]{ApisaWright} for a discussion). In this case, $\cM$ is $h$-geminal up to marked points if $\cM^{pt}$ is $h$-geminal. 
\end{defn}

We begin by showing that adding all the periodic points to an $h$-geminal subvariety of rank at least two does not change the fact that it is $h$-geminal.

\begin{lem}\label{L:AddingPPtoGeminal}
If $\cM$ is an $h$-geminal subvariety of rank at least two, then $\cM^{pt}$ is also $h$-geminal. 
\end{lem}
\begin{proof}
By Theorem \ref{T:geminal2} \eqref{I:geminal2:WhatIsM}, $\cM$ is either a nonhyperelliptic component of a stratum of Abelian differentials or a locus of covers of a genus zero quadratic double. In the first case there are no periodic points by Theorem \ref{T:PeriodicPoints-NonHyp}. On a quadratic double of a genus zero stratum, the only periodic points are fixed points of the hyperelliptic involution by Theorem \ref{T:PeriodicPointsHyp}. After marking such points, the locus is still a quadratic double of a genus zero stratum. The result now follows since a full locus of good covers of a genus zero quadratic double is $h$-geminal.
\end{proof}

\begin{lem}\label{L:RemovingMarkedPointsStillOptimal}
Let $f: (X, \omega) \ra (Y, \eta)$ be a good translation cover. Fix a marked point $q \in (Y, \eta)$. Let $\For'(X, \omega)$ denote $(X ,\omega)$ with a subset of the points in $f^{-1}(q)$ removed. Let $\For'(Y, \eta)$ denote $(Y, \eta)$ where, if every preimage of $q$ is removed, $q$ is removed, and where $\For'(Y, \eta) = (Y, \eta)$ otherwise. We adopt the convention that a zero can never be unmarked. Then $f: \For'(X, \omega) \ra \For'(Y, \eta)$ is good. 
\end{lem}
\begin{proof}
By assumption, the preimage under $f$ of a cylinder $C \subseteq (Y, \eta)$ is a single cylinder. If some of the preimages of $q$ remain marked on $\For'(X, \omega)$, then $f^{-1}(C)$ remains a cylinder on $\For'(X, \omega)$. In other words, the only way that $f: \For'(Y, \eta) \ra \For'(X, \omega)$ could fail to be good is if every preimage of $q$ has been removed (in particular, $q$ is not a branch point); we make this assumption now. 

It is easy to see that if $C$ is a cylinder on $(Y, \eta)$ that continues to be a cylinder on $\For'(Y, \eta)$ then its preimage on $\For'(X, \omega)$ continues to be a single cylinder. 

We must therefore consider the situation where $C$ ceases to be a cylinder on $\For'(Y, \eta)$, i.e. where, on $(Y, \eta)$, one of its boundary components consists of a saddle connection $s$ joining $q$ to itself. Since $q$ is a marked point, there is another cylinder $C'$ on $(Y, \eta)$ that also has $s$ as one of the  components of its boundary. While $C$ is no longer a cylinder on $\For'(Y, \eta)$, $C \cup C'$ is. On $(X, \omega)$ the preimage of $C \cup C'$ is a union of two cylinders $D \cup D'$ both of which share a boundary component $f^{-1}(s)$. Since $q$ is not a branch point, $f^{-1}(s)$ is just a collection of saddle connections joining marked points in the preimage of $q$. In particular, once we forget all preimages of $q$, $D \cup D'$ becomes a single cylinder. 

Since every cylinder on $\For'(Y, \eta)$ was either a cylinder on $(Y, \eta)$ or contains a cylinder on $(Y, \eta)$ that ceases to be one on $\For'(Y, \eta)$, we are done. 
\end{proof}

\begin{cor}\label{C:MinIsOpt}
Let $\cM$ be a rank at least two invariant subvariety that is $h$-geminal up to marked points. If $(X, \omega)$ has dense orbit in $\cM$, then $\pi_{X_{min}}$ is $\cM$-optimal.
\end{cor}
\begin{proof}
Let $(X, \omega)^{pt}$ denote $(X, \omega)$ with all of its $\cM$-periodic points marked. By definition, its orbit closure is $\cM^{pt}$, which is $h$-geminal by assumption. By Theorem \ref{T:geminal2} \eqref{I:geminal2:PiOpt}, $\pi_{X_{min}}$ is an optimal map when its domain is taken to be $(X, \omega)^{pt}$. Therefore, $\pi_{X_{min}}$ is a good map even when its domain is taken to be $(X, \omega)$ (by Lemma \ref{L:RemovingMarkedPointsStillOptimal}). The map is optimal by Theorem \ref{T:MinimalCover}. $\cM$-optimality follows by Remark \ref{R:WhyWeCanPerturb}.
\end{proof}

\begin{cor}\label{C:GeneralOptDegeneratesToOpt}
Suppose that $\bfC$ is a collection of generic $\cM$-equivalent cylinders on a surface $(X, \omega)$ in an invariant subvariety $\cM$ so that $\sigma_{\bfC} \in T_{(X, \omega)} \cM$ and such that one of the following holds:
\begin{enumerate}
    \item\label{I:R1} $\cM$ has rank one and minimal homological dimension.
    \item\label{I:R2} $\cM$ has rank at least two and is $h$-geminal up to marked points.
\end{enumerate}
Let $\pi_{opt}$ be the $\cM$-optimal map on $(X, \omega)$. If $\cM_{\bfC}$ has dimension one less than that of $\cM$, then $\Col_{\bfC}(X, \omega)$ is connected and $\Col_{\bfC}(\pi_{opt})$ is $\cM_{\bfC}$-optimal. Moreover, if $\cM$ has rank two, then $\cM_{\bfC}$ is $h$-geminal up to marked points.
\end{cor}
\begin{proof}
When $\cM$ has rank one this claim is immediate from Theorem \ref{T:geminal1}. Therefore, assume that $\cM$ has rank at least two and that it is $h$-geminal up to marked points. 

Let $\bfC^{pt}$ denote the cylinders corresponding to $\bfC$ on $(X, \omega)^{pt}$, which we will use to denote the surface $(X, \omega)$ with all of its periodic points marked. Assume without loss of generality that $\bfC$ is horizontal.

\begin{sublem}
There is an arbitrarily small purely imaginary $v \in \mathrm{Twist}(\bfC, \cM)$ so that after replacing $(X, \omega)$ with $(X, \omega) + v$, the cylinders in $\bfC^{pt}$ are $\cM^{pt}$-generic.
\end{sublem}
\begin{proof}
If $\cM$ is a nonhyperelliptic component of a stratum, then $\bfC = \bfC^{pt}$ (by Theorem \ref{T:PeriodicPoints-NonHyp}) and there is nothing to show. Otherwise, $\cM$ is a full locus of covers of a hyperelliptic locus (by Theorem \ref{T:geminal2} \eqref{I:geminal2:WhatIsM}). 

Let $\bfD$ (resp. $\bfD^{pt}$) be the cylinders that are the image of $\bfC$ (resp. $\bfC^{pt}$) under $\pi_{opt}$. It follows from Theorem \ref{T:PeriodicPointsHyp} that, since the cylinders in $\bfD$ are generic, the cylinders in $\bfD^{pt}$ fail to be generic if and only if there are two free points $p$ and $p'$ on the boundary of cylinders in $\bfD$ so that, letting $J$ denote the hyperelliptic involution, $p$ and $J(p')$ lie on the same horizontal leaf. It is always possible to move free points along vertical leaves to avoid this. 
\end{proof}

Notice that using the previous sublemma to replace $(X, \omega)$ with $(X, \omega) + v$ does not affect the surface $\Col_{\bfC}(X, \omega)$ since $v$ was a purely imaginary element of $\mathrm{Twist}(\bfC, \cM)$ and since, to form $\Col_{\bfC}(X, \omega)$ we vertically collapse the cylinders in $\bfC$. Therefore, assume without loss of generality that $\bfC^{pt}$ consists of $\cM^{pt}$-generic cylinders.

\begin{sublem}\label{SL:DimHypothesis}
$\cM_{\bfC}$ is $h$-geminal up to marked points. Moreover,
\[ \mathrm{rank}(\cM_{\bfC}) = \mathrm{rank}(\cM_{\bfC^{pt}}^{pt}) \qquad \text{and} \qquad \dim \cM_{\bfC^{pt}}^{pt} = \dim \cM^{pt} - 1. \]
\end{sublem}
\begin{proof}
Since $\Col_{\bfC}(X, \omega)$ is simply $\Col_{\bfC^{pt}}\left( (X, \omega)^{pt} \right)$ with some marked points forgotten, we have that $\dim \cM^{pt}_{\bfC^{pt}} \geq \dim \cM_{\bfC}$ and $\mathrm{rank}(\cM_{\bfC}) = \mathrm{rank}(\cM_{\bfC^{pt}}^{pt})$. By definition of periodic points, and since the dimension always decreases when passing to the boundary, we have the following,
\[ \dim \cM - 1 = \dim \cM_{\bfC} \leq \dim \cM^{pt}_{\bfC^{pt}} \leq \dim \cM^{pt} - 1 = \dim \cM -1.  \]
This shows that $\dim \cM^{pt}_{\bfC^{pt}} = \dim \cM_{\bfC}$ and hence that the marked points forgotten in the passage from $\Col_{\bfC^{pt}}\left( (X, \omega)^{pt} \right)$ to $\Col_{\bfC}(X, \omega)$ were periodic points. In particular, $\cM_{\bfC}$ is $h$-geminal up to marked points since $\cM^{pt}_{\bfC^{pt}}$ is $h$-geminal by Apisa-Wright \cite[Lemma 4.4]{ApisaWrightGeminal}.
\end{proof}

\begin{sublem}\label{SL:MHD+Degen}
$\cM_{\bfC}$ and $\cM_{\bfC^{pt}}^{pt}$ consist of connected surfaces, have minimal homological dimension, and $\Col_{\bfC^{pt}}(\pi_{opt})$ is $\cM_{\bfC^{pt}}^{pt}$-optimal.
\end{sublem}
\begin{proof}
Since $\cM^{pt}$ is $h$-geminal by assumption, it follows that it has minimal homological dimension (this follows from Theorem \ref{T:geminal2} \eqref{I:geminal2:WhatIsM}). Therefore, $\Col_{\bfC^{pt}}((X, \omega)^{pt})$ is connected and  $\cM_{\bfC^{pt}}^{pt}$ has minimal homological dimension (by Corollary \ref{C:MHD-Boundary:Connected}). The same holds for $\Col_{\bfC}(X, \omega)$ and $\cM_{\bfC}$ respectively. By Theorem \ref{T:geminal2} \eqref{I:geminal2:PiOptDegeneratesToPiOpt}, $\Col_{\bfC^{pt}}(\pi_{opt})$ is $\cM^{pt}_{\bfC^{pt}}$-optimal (note that the condition that $\dim \cM_{\bfC^{pt}}^{pt} = \dim \cM^{pt} - 1$ is satisfied by Sublemma \ref{SL:DimHypothesis}). Note that in this proof the application of Theorem \ref{T:geminal2} and Corollary \ref{C:MHD-Boundary:Connected} required that $\bfC^{pt}$ consisted of generic cylinders.
\end{proof}

It remains to see that $\Col_{\bfC}(\pi_{opt})$ is $\cM_{\bfC}$-optimal. Suppose first that $\cM_{\bfC}$ has rank one. By Sublemma \ref{SL:MHD+Degen} and Theorem \ref{T:geminal1}, $\Col_{\bfC^{pt}}(\pi_{opt}) = \pi_{abs}$. Therefore, $\Col_{\bfC}(\pi_{opt}) = \pi_{abs}$, which is the $\cM_{\bfC}$-optimal map (by Sublemma \ref{SL:MHD+Degen} and Theorem \ref{T:geminal1}).

Suppose finally that $\cM_{\bfC}$ has rank two. By Sublemma \ref{SL:MHD+Degen} and Theorem \ref{T:geminal2} \eqref{I:geminal2:PiOpt}, when $\Col_{\bfC^{pt}}(\pi_{opt})$ is deformed to a nearby surface with dense orbit it is the minimal cover. The same must hold for $\Col_{\bfC}(\pi_{opt})$ and hence $\Col_{\bfC}(\pi_{opt})$ must be $\cM_{\bfC}$-optimal (by Corollary \ref{C:MinIsOpt} and the fact that we have established that $\cM_{\bfC}$ is $h$-geminal up to marked points and of rank at least two). 
\end{proof}

\subsection{General degenerations of optimal maps}

At this point we have shown that, if $\cM$ is $h$-geminal up to marked points, then the surfaces in $\cM$ possess $\cM$-optimal maps that are well-behaved with respect to cylinder degenerations performed using the standard shear. We now wish to study what happens to these optimal maps under more general cylinder degenerations. 

Throughout this subsection we will assume that $\bfC$ is an equivalence class of generic cylinders on a surface $(X, \omega)$ in an invariant subvariety $\cM$ and that $v \in \mathrm{Twist}(\bfC, \cM)$ defines a cylinder degeneration.

\begin{defn}
Say that $v$ is \emph{orthogonal} if, after rotating $\bfC$ so that it is horizontal, $v$ is purely imaginary.  
\end{defn}

Recall that $\bfC_v$ was defined in Definition \ref{D:Collapse}.

\begin{defn}
Suppose that $\pi: (X, \omega) \ra (Y, \eta)$ is a good map and that $v$ is orthogonal. By definition, for each cylinder $C \in \bfC$ there is a cylinder $D \subseteq (Y, \eta)$ whose full preimage under $\pi$ is $C$. If $v = \sum_{C \in \bfC} a_C \gamma_C^*$ where $a_C \in \mathbb{C}$ and $\gamma_C$ denotes the core curve of $C$, then let $\pi(v) := \sum_{C \in \bfC} a_C \gamma_{\pi(C)}^*$. It is obvious that $(X, \omega) + tv$ remains a translation cover of $(Y, \eta) + t\pi(v)$. Both of these paths converge to surfaces $(X', \omega')$ and $(Y', \eta')$ by Apisa-Wright \cite[Lemma 4.9]{ApisaWrightHighRank} as $t$ approaches $t_v$ (see Definition \ref{D:Collapse} for the definition of $t_v$). Notice that setting $v' := \sum_{C \notin \bfC_v} a_C \gamma_C^*$ (and using that $v$ is orthogonal) we have
\[ (X', \omega') = \Col_{\bfC_v}( (X, \omega) + v' ) \quad \text{and} \quad (Y', \omega') = \Col_{\pi(\bfC)_{\pi(v)}}( (Y, \eta) + v' ). \]
Since passing from $(X, \omega)$ to $(X, \omega) + v'$ simply involves applying an invertible linear map to cylinders in $\bfC - \bfC_v$, it is clear that $\pi$ induces a map 
\[ \pi': (X, \omega) + v' \ra (Y, \eta) + \pi(v'). \]
Define a map 
\[ \Col_v(\pi): \Col_v(X, \omega) \ra \Col_{\pi(v)}(Y, \eta). \]
by setting $\Col_v(\pi) := \Col_{\bfC_v}(\pi')$.
%
% The surface $(X', \omega')$ can be formed by taking each cylinder in $C$ and either applying the linear map specified by $a_C \gamma_C^*$ to it (if $C$ does not have its height go to zero) or collapsing $C$ (if it does). The surface $(Y', \eta')$ is formed similarly. By Lemma \ref{L:DefinitionCol(f)}, it follows that there is a map
% \[ \Col_v(\pi): \Col_v(X, \omega) \ra \Col_{\pi(v)}(Y, \eta) \]
% This map is constructed by thinking about the passage from $(X, \omega)$ to $(X', \omega')$ as a sequence of successive cylinder collapses and then repeatedly applying Lemma \ref{L:DefinitionCol(f)}.
\end{defn}

\begin{lem}\label{L:TwistWithCollapse}
Suppose that $v$ is orthogonal and that $\overline{\bfC} \ne (X, \omega)$. If $\Col_v(X, \omega)$ is connected, then there is an arbitrarily small element $w \in \mathrm{Twist}(\bfC - \bfC_v, \cM)$ so that $\Col_v(X, \omega)$ and $\Col_{\bfC}\left( (X, \omega) + w \right)$ belong to the same component of the boundary.
\end{lem}

The connectedness is solely used to apply the cylinder deformation theorem to surfaces in $\cM_v$.

\begin{proof}
Without loss of generality, suppose that $\bfC$ is horizontal and $v$ is purely imaginary. If $\bfC_v = \bfC$ there is nothing to show, so suppose that this is not the case. Since $\Col_v(\bfC)$ remains an equivalence class of cylinders, the standard shear is an element of $T_{\Col_v(X, \omega)} \cM_v$ by the cylinder deformation theorem. Let $w$ be the corresponding element in $\mathrm{Twist}(\bfC - \bfC_v, \cM)$. For almost every real $a$, $(X, \omega) + aw$ contains no vertical saddle connections in $\overline{\bfC}$ except for those in $\overline{\bfC_v}$. Notice that $\Col_v(X, \omega)$ is in the same component of the boundary as
\[ \Col_v(X, \omega) + a\sigma_{\Col_v(\bfC)} = \Col_v((X, \omega) + aw). \]
Since $\Col_v((X, \omega) + aw)$ contains no vertical saddle connections in $\overline{\Col_v(\bfC)}$ (by our choice of $a$), it belongs to the same component of the boundary as 
\[ \Col_{\Col_v(\bfC)} \Col_v((X, \omega) + aw) = \Col_{\bfC}\left( (X, \omega) + aw \right). \]
This shows that $\Col_v(X, \omega)$ and $\Col_{\bfC}\left( (X, \omega) + aw \right)$ belong to the same component of the boundary, as desired.
%
% For any such $a$, it is easy to see that 
% \[ \Col_v((X, \omega) + aw) = \Col_v((X, \omega)) + aw. \]
% Since $\Col_v(\bfC)$ contains no vertical saddle connections on $\Col_v((X, \omega)+aw)$, $\Col_v((X, \omega)+aw)$ belongs to the same stratum as
% \[ \Col_{\Col_v(\bfC)} \Col_v((X, \omega) + aw) = \Col_{\bfC}\left( (X, \omega) + aw \right). \]
% This shows that $\Col_v(X, \omega)$ and $\Col_{\bfC}\left( (X, \omega) + aw \right)$ belong to the same component of the boundary, as desired.
\end{proof}

\begin{lem}\label{L:PiOptWithTwist}
Suppose that $v$ is orthogonal and that $\overline{\bfC} \ne (X, \omega)$. Suppose that $\cM_{\bfC} = \cM_v$ and that $\Col_v(X, \omega)$ is connected. Suppose finally that $\pi$ is a translation cover with domain $(X, \omega)$. If $\Col_{\bfC}(\pi)$ is $\cM_{\bfC}$-optimal, then so is $\Col_v(\pi)$.
\end{lem}
\begin{proof}
Suppose without loss of generality that $\bfC$ is horizontal. If $\bfC_v = \bfC$ there is nothing to prove so suppose that this is not the case. Let $\gamma(t) = \Col_v(X, \omega) - it\sigma_{\Col_v(\bfC)}$, which is a path in $\cM_v$ (by the cylinder deformation theorem) so that $\gamma(1) = \Col_{\bfC}(X, \omega)$. By assumption, $\Col_{\bfC}(\pi)$ is $\cM_{\bfC}$-optimal, so it is possible to deform it along the path so that it becomes a translation cover $\pi'$ on $\Col_v(X, \omega)$. Since $\pi'$ and $\Col_v(\pi)$ are holomorphic maps that have the same fibers on the open set $\Col_v(X, \omega) - \Col_v(\bfC)$, they must coincide and so $\Col_v(\pi)$ is $\cM$-optimal.  
\end{proof}

%%%%%%%%%%%%%%%%%%%%%%%
%
% SECTION - PROOF OF MAIN THEOREM
%
%
%%%%%%%%%%%%%%%%%%%%%%%

\section{Proof of Theorem \ref{T:main}}

This section is devoted to the proof of Theorem \ref{T:main}. We proceed by induction using Proposition \ref{P:BaseBase} as the base case. The following assumption is the inductive hypothesis.

\begin{ass}\label{A:IH}
Suppose that $\cM$ is an invariant subvariety of minimal homological dimension that is rank at least two and not rank two rel zero. Any smaller dimensional invariant subvariety of minimal homological dimension and rank at least two is $h$-geminal up to marked points. 
\end{ass}

\noindent Our goal is to show that $\cM$ is $h$-geminal up to marked points.

\begin{lem}\label{L:MinIsOptimal}
If $(X, \omega)$ has dense orbit in $\cM$, then $\pi_{X_{min}}$ is $\cM$-optimal. 

Suppose additionally that $\bfC$ is a generic equivalence class of cylinders with $v \in \mathrm{Twist}(\bfC, \cM)$ defining a typical orthogonal cylinder degeneration. Then $\cM_v$ has minimal homological dimension, $\Col_v(X, \omega)$ is connected, and $\Col_v(\pi_{X_{min}})$ is $\cM_v$-optimal. 
\end{lem}
\begin{proof}
Throughout this proof it will be important to keep Remark \ref{R:WhyWeCanPerturb} in mind.

Let $\bfC_1$ be an equivalence class of cylinders on $(X, \omega)$. After perhaps perturbing we may assume that the cylinders in $\bfC_1$ are horizontal and form a generic equivalence class. After perturbing again (using Corollary \ref{C:FieldDefinition} and Lemma \ref{L:Stable} \eqref{I:PerturbToStableQ}), suppose that $(X, \omega)$ is cylindrically stable. 

If $\bfC_1$ is involved with rel (see Definition \ref{D:InvolvedWithRel}) or if the rank of $\cM$ is at least three, then let $\bfC_2$ be any other equivalence class of horizontal cylinders. If $\cM$ has rank two and $\bfC_1$ is not involved with rel then let $\bfC_2$ be any equivalence class of horizontal cylinders involved with rel (such an equivalence class of cylinders exists since, when $\cM$ has rank two, it has positive rel, by Assumption \ref{A:IH}, and so at least one horizontal equivalence class of cylinders is involved with rel by definition of cylindrical stability). Perturb $(X, \omega)$ once more so that $\bfC_1$ and $\bfC_2$ become generic equivalence classes and so $(X, \omega)$ has dense orbit in $\cM$. 

For $i \in \{1, 2\}$, there is a $v_i \in \mathrm{Twist}(\bfC_i, \cM)$ that defines a typical cylinder degeneration (by Theorem \ref{T:Typical}). Without loss of generality, after perhaps applying the standard shear to $\bfC_1$ and $\bfC_2$, we may suppose that $v_1$ and $v_2$ are typical orthogonal degenerations (see Remark \ref{R:TypicalOrthogonalConstruction}). By Lemma \ref{L:TwistWithCollapse}, we may suppose, again after perhaps perturbing, that $\cM_{v_i} = \cM_{\bfC_i}$. Notice that the following hold: 
\begin{enumerate}
    \item\label{I:MHD} $\cM_{v_i}$ has minimal homological dimension and consists of connected surfaces by Corollary \ref{C:MHD-Boundary:Connected}.
    \item\label{I:Dim} $\cM_{v_i}$ has codimension one and $\mathrm{rank}(\cM_{v_i}) = \mathrm{rank}(\cM)$ if $\bfC_i$ is involved with rel and $\mathrm{rank}(\cM_{v_i}) = \mathrm{rank}(\cM)-1$ otherwise, by Theorem \ref{T:Typical}.
    \item\label{I:GeminalUpToMP} If $\cM_{v_i}$ has rank at least two, then it is $h$-geminal up to marked points and any surface $(X, \omega)$ with dense orbit in $\cM_{v_i}$ has the property that $\pi_{X_{min}}$ is $\cM_{v_i}$-optimal (by \eqref{I:MHD}, Assumption \ref{A:IH}, and Corollary \ref{C:MinIsOpt}). Note that $\cM_{v_i}$ has rank at least two when  $\cM$ has rank at least three and also when $\bfC_i$ is involved with rel by \eqref{I:Dim}. Our choice of $\bfC_2$ implies that this occurs for $\cM_{v_1}$ or $\cM_{v_2}$ and possibly for both.
    \item\label{I:OptimalMaps} On $\Col_{\bfC_i}(X, \omega)$ there is an $\cM_{\bfC_i}$-optimal map $\pi_i$ (by \eqref{I:MHD} and Theorem \ref{T:geminal1} when $\cM_{v_i}$ has rank one and by \eqref{I:GeminalUpToMP} when $\cM_{v_i}$ has rank at least two).
    \item\label{I:DegeneratingOptimal} Since $\Col_{\bfC_{i+1}}(\bfC_i)$ remains a collection of generic equivalent cylinder with $\cM_{\bfC_1, \bfC_2}$ having dimension one less than $\cM_{\bfC_i}$ (by Theorem \ref{T:Typical}), $\Col_{\Col_{\bfC_i}(\bfC_{i+1})}(\pi_i)$ is the unique $\cM_{\bfC_1, \bfC_2}$-optimal map on $\ColOneTwoX$ (by Corollary \ref{C:GeneralOptDegeneratesToOpt}).
\end{enumerate}  

In light of \eqref{I:DegeneratingOptimal}, by the diamond lemma (Apisa-Wright \cite[Lemma 2.3]{ApisaWrightDiamonds}), there is a map $\pi$ on $(X, \omega)$ so that $\Col_{\bfC_i}(\pi) = \pi_i$. Note that the condition in the diamond lemma that $\pi_i^{-1}\left( \pi_i\left( \Col_{\bfC_i}(\bfC_{i+1}) \right) \right) = \Col_{\bfC_i}(\bfC_{i+1})$ follows by definition of $\cM_{\bfC_i}$-optimal. 

By \eqref{I:GeminalUpToMP} there is some $j \in \{1, 2\}$ so that $\pi_j$ is $\cM_{\bfC_j}$-generic and becomes $\pi_{Y_{min}}$ for any nearby surface $(Y, \eta)$ with dense orbit in $\cM_{\bfC_j}$. Since $\pi_{X_{min}}$ is $\cM$-generic (since $(X, \omega)$ has dense orbit in $\cM$), it follows that $\Col_{\bfC_j}(\pi_{X_{min}})$ is $\cM_{\bfC_j}$-generic (it suffices to note that a neighborhood of $\Col_{\bfC_j}(X,\omega)$ can be obtained by perturbing $(X, \omega)$ and then collapsing $\bfC_j$. This follows from Mirzakhani-Wright \cite[Proposition 2.6]{MirWri} together with the main results of Mirzakhani-Wright and Chen-Wright \cite{MirWri, ChenWright}.) If $\pi \ne \pi_{X_{min}}$, then $\Col_{\bfC_j}(\pi_{X_{min}})$ has larger degree than $\pi_j$. So on a nearby surface $(Y, \eta) \in \cM_{\bfC_j}$ with dense orbit the map corresponding to $\Col_{\bfC_j}(\pi_{X_{min}})$ would have larger degree than $\pi_{Y_{min}}$. This contradicts Theorem \ref{T:MinimalCover} and shows that $\pi = \pi_{X_{min}}$. 

Since $\bfC_1$ is arbitrary and each cylinder in it is the full preimage of its image under $\pi$ (since this is the case for $\ColTwo(\bfC_1)$ with respect to $\pi_2$), it follows that $\pi_{X_{min}}$ is $\cM$-optimal. Since $v_1$ was an arbitrary typical orthogonal cylinder degeneration of $\bfC_1$ and since we arranged for $\cM_{v_1} = \cM_{\bfC_1}$, the second claim follows from Lemma \ref{L:PiOptWithTwist}.
\end{proof}

We now proceed with the proof that $\cM$ is $h$-geminal up to marked points. Fix a surface $(X, \omega)$ with dense orbit in $\cM$ and let $(X', \omega')$ denote $(X, \omega)$ with all of its periodic points marked. Let $\cM_{min}$ (resp. $\cM_{min}'$) denote the orbit closure of $(X_{min}, \omega_{min})$ (resp. $(X_{min}', \omega_{min}')$). Notice that these surfaces first appear in Theorem \ref{T:MinimalCover}.

% We may assume without loss of generality that $\cM$ is not geminal (by Lemma \ref{L:AddingPPtoGeminal}). In particular, this means that we may assume without loss of generality that there is an equivalence class $\bfC_1$ of cylinders on $(X, \omega)$ that cannot be partitioned into subequivalence classes of at most two cylinders, all of which are isometric (we will call such a partition a partition into \emph{geminal subequivalence classes}).

% Let $\cM_{min}$ denote the orbit closure of $(X_{min}, \omega_{min}) := \pi_{X_{min}}\left( X, \omega \right)$. Let $\cM_{min}'$ be the orbit closure of $\pi_{X_{min}}\left( (X, \omega; P) \right)$.

Recall the convention that if $C$ is a cylinder we will let $\gamma_C^*$ denote the intersection pairing with its core curve. 

\begin{lem}\label{L:FullRankMeansDone}
If $\cM_{min}$ has full rank, i.e. has rank equal to that of the stratum of Abelian differentials containing it, then $\cM$ is $h$-geminal up to marked points. 
\end{lem}
\begin{proof}
If $\cM_{min}$ has full rank, then by Mirzakhani-Wright \cite[Theorem 1.1]{MirWri2}, $\For(\cM_{min})$ is either a nonhyperelliptic component of a stratum of Abelian differentials or a hyperelliptic locus (possibly equal to the full component) of a stratum. In the first case (by Theorem \ref{T:PeriodicPoints-NonHyp}), $\cM_{min}'$ is itself a stratum and so every cylinder is free. In the second (by Theorem \ref{T:PeriodicPointsHyp}), $\cM_{min}'$ is a hyperelliptic locus with every fixed point of the hyperelliptic involution marked and with all other marked points occurring in pairs exchanged by the hyperelliptic involution. In particular, $\cM_{min}'$ is a quadratic double and hence $h$-geminal. Since $\pi_{X_{min}}$ is optimal (by Lemma \ref{L:MinIsOptimal}), every cylinder on $(X', \omega')$ contains exactly one cylinder in its preimage. This shows that the orbit closure of $(X', \omega')$ is $h$-geminal and hence that $\cM$ is $h$-geminal up to marked points.  
\end{proof}

Our goal is therefore to show that $\cM_{min}$ has full rank. The following is a useful test to determine when this is the case. For ease of notation we will let $\pi := \pi_{X_{min}}$.

\begin{lem}\label{L:SimpleCylinder}
Suppose that $\bfC$ is a generic equivalence class of cylinders on $(X, \omega)$. Suppose too that $C$ is a cylinder in $\bfC$ so that $\pi(C)$ has a boundary comprised of a single saddle connection. Suppose finally that $v \in \mathrm{Twist}\left( \bfC, \cM \right)$ determines an orthogonal cylinder degeneration with $\bfC_v = \{C\}$. Then $\cM_{min}$ has full rank.
\end{lem}
\begin{proof}
Without loss of generality (after perhaps rotating and shearing), suppose that $\bfC$ consists of horizontal cylinders and that $\pi(C)$ contains a vertical saddle connection $s$. If $v = \sum_{D \in \bfC} a_D \gamma_D^*$, then let $w := \sum_{D \in \bfC} a_D \gamma_{\pi(D)}^*$. Notice that $w$ is a purely imaginary rel vector in $T_{\pi(X, \omega)} \cM_{min}$ and $(\pi(\bfC))_w = \{\pi(C)\}$. 

It is easy to see that, since $\pi(C)$ has one boundary consisting of a single saddle connection, that $\Col_w(X, \omega)$ is connected and that the collection of vanishing cycles is just the line generated by $s$. For instance the second claim follows from Apisa-Wright \cite[Lemma 4.25]{ApisaWrightHighRank}. Let $\cH$ be the component of the stratum of Abelian differentials containing $\pi(X, \omega)$. By Apisa-Wright \cite[Lemma 3.8]{ApisaWrightHighRank} since $w$ is a rel vector that evaluates nontrivially on $s$, we have the following:
\begin{equation}\label{E:SimpleRankEquation} \mathrm{rank}((\cM_{min})_w) = \mathrm{rank}(\cM_{min}) \qquad \text{and} \qquad \mathrm{rank}(\cH_w) = \mathrm{rank}(\cH). 
\end{equation}

By Lemma \ref{L:MinIsOptimal}, since $v$ is a typical orthogonal degeneration, $\cM_v$ has minimal homological dimension and consists of connected surfaces. By Equation \ref{E:SimpleRankEquation} (which implies that $\mathrm{rank}(\cM_v) = \mathrm{rank}(\cM)$) and Assumption \ref{A:IH}, $\cM_v$ is $h$-geminal up to marked points and the $\cM_v$-optimal map coincides with the minimal cover on any surface in $\cM_v$ with dense orbit (by Corollary \ref{C:MinIsOpt}). Therefore, by Lemma \ref{L:MinIsOptimal} we have the following,
\[ (\cM_{min})_w = (\cM_{v})_{min}. \]
Moreover, since $\cM_v$ is $h$-geminal up to marked points, $(\cM_v)_{min}$ has full rank (by Theorem \ref{T:geminal2} \eqref{I:geminal2:WhatIsM}). Since $\left( \cM_{min} \right)_w$ has the same rank as $\cH_w$, the same is true of $\cM_{min}$ and $\cH$, by Equation \eqref{E:SimpleRankEquation}, as desired. 
%
% It is easy to see that, since $\pi(C)$ has one boundary consisting of a single saddle connection, that $\Col_w(X, \omega)$ is connected and that the collection of vanishing cycles is just the line generated by $s$. For instance the second claim follows from Apisa-Wright \cite[Lemma 4.25]{ApisaWrightHighRank}. Letting $\cH$ be the component of the stratum of Abelian differentials containing $\pi(X, \omega)$ this implies the following:
% \[ \dim (\cM_{min})_w = \dim \cM_{min} - 1 \qquad \text{and} \qquad \dim \cH_w = \dim \cH - 1. \]
% By Apisa-Wright \cite[Lemma 3.8]{ApisaWrightHighRank} since $w$ is a rel vector that evaluates nontrivially on $s$, we have the following:
% \begin{equation}\label{E:SimpleRankEquation} \mathrm{rank}((\cM_{min})_w) = \mathrm{rank}(\cM_{min}) \qquad \text{and} \qquad \mathrm{rank}(\cH_w) = \mathrm{rank}(\cH). 
% \end{equation}
% By Lemma \ref{L:MinIsOptimal},
% \[ (\cM_{min})_w = (\cM_{v})_{min}. \]
% By Assumption \ref{A:IH} $\left( \cM_{min} \right)_w$ has full rank in $\cH_w$ and hence the same is true of $\cM_{min}$ in $\cH$, by Equation \eqref{E:SimpleRankEquation}, as desired. 
\end{proof}

We will now rephrase the preceding result into the form in which we will apply it. 

\begin{cor}\label{C:SimpleCylinder}
Suppose that $\bfC$ is a generic equivalence class containing at least two horizontal cylinders on $(X, \omega)$. Suppose too that $w$ is a typical orthogonal degeneration in $\mathrm{Twist}(\bfD, \cM)$ where $\cM_w$ has rank at least two and where $\bfD$ is an equivalence class either equal to $\bfC$ or disjoint and not parallel to it.  Finally suppose that one of the following happens:
\begin{enumerate}
    \item\label{I:test:free-cyl} There is a cylinder $C$ in $\bfC$ so that $\Col_w(\pi)(\Col_w(C))$ is free, or 
    \item\label{I:test:rel} There are cylinders of generically equal circumference $C_1$ and $C_2$ in $\bfC$ so that $\gamma_{C_1}^* - \gamma_{C_2}^* \in T_{(X, \omega)} \cM$, or 
    \item\label{I:test:free-pt} There is a cylinder $C$ in $\bfC$ so that $\Col_w(\pi)(\Col_w(C))$ contains a free marked point on its boundary. 
\end{enumerate}
Then $\cM$ is $h$-geminal up to marked points.
%
% Suppose that $\bfC$ is a generic equivalence class containing at least two horizontal cylinders on $(X, \omega)$. Suppose too that $w$ is a typical orthogonal degeneration in $\mathrm{Twist}(\bfD, \cM)$ where $\cM_w$ has rank at least two and where $\bfD$ either equal to $\bfC$ or disjoint and not parallel to $\bfC$.  Finally suppose that $C$ is a cylinder in $\bfC$ so that either $\Col_w(\pi)(\Col_w(C))$ is free or so that it contains a free marked point on its boundary. Then $\cM$ is $h$-geminal up to marked points.
\end{cor}
\begin{proof}
% Note that $\cM_w$ has minimal homological dimension (by Lemma \ref{L:MinIsOptimal}) and that $\cM_w$ and $\cM$ have the same rank (by Apisa-Wright \cite[Lemma 11.4]{ApisaWrightHighRank}), so $\cM_w$ is either a stratum of Abelian differentials or a full locus of translation covers of a hyperelliptic locus (by Assumption \ref{A:IH}). 
%
%We begin by noting that, after rotating so that $\bfD$ is horizontal, to take the imaginary part of $w$ and (after perhaps shearing $\bfD$), assume that $w$ is purely imaginary. We will make this assumption. 

Note that $\cM_w$ is either a stratum of Abelian differentials or a full locus of translation covers of a hyperelliptic locus (by Lemma \ref{L:MinIsOptimal}, Assumption \ref{A:IH}, and Theorem \ref{T:geminal2} \eqref{I:geminal2:WhatIsM}). In the second case, the cover is given by $\Col_w(\pi)$ by Lemma \ref{L:MinIsOptimal}. 

Since the cylinders in $\bfC$ are generic, the same holds for those in $\Col_w(\bfC)$. This is obvious in the case that $\bfD$ is disjoint and non-parallel to $\bfC$ and follows from Theorem \ref{T:Typical} in the case that $\bfC = \bfD$. Since $\Col_w(\pi)\left( \Col_w(\bfC) \right)$ is a collection of generic cylinders we see that those cylinders are all simple when $\cM_w$ is a stratum and consists of simple and half-simple cylinders when $\cM_w$ is a full locus of covers of a hyperelliptic locus (by Corollary \ref{C:PeriodicPointsHyp}). It follows that for each $C \in \bfC$, $\pi(C)$ is simple or half simple, since $\Col_w(\pi)(\Col_w(C))$  has at least as many saddle connections on each boundary component as $\pi(C)$. 

By Lemmas \ref{L:FullRankMeansDone} and \ref{L:SimpleCylinder} it suffices to show that there is a rel deformation that increases the height of one cylinder in $\bfC$ while not increasing the height of all other cylinders in $\bfC$. 

Suppose first that $C \in \bfC$ is a cylinder so that $\Col_w(\pi)(\Col_w(C))$ is free. It follows that $\pi(C)$ and hence $C$ itself are also free. By Mirzakhani-Wright (Theorem \ref{T:CylECTwistSpace}), since $\bfC \ne \{C\}$, there is a positive real number $a$ so that $v : = i\left( -\gamma_C^* + a \sigma_{\bfC} \right)$ is rel. This is the desired rel deformation. 

The second claim is immediate since $\gamma_{C_1}^* - \gamma_{C_2}^*$ is rel by Apisa-Wright \cite[Corollary 11.5]{ApisaWrightHighRank}.

% Suppose second that there are cylinders $C_1$ and $C_2$ in $\bfC$ so that if $\gamma_i$ denotes the core curve of $\Col_w(\pi)(\Col_w(C_i))$, then $\gamma_1^* - \gamma_2^*$ is tangent to $(\cM_w)_{min}$. It follows that 

Suppose now that $C \in \bfC$ is a cylinder so that $\Col_w(\pi)(\Col_w(C))$ contains a free point on its boundary. Since this cylinder is generic, the boundary containing the free point must consist of a saddle connection $s$ joining the free marked point to itself. In particular, there is another cylinder $D \in \bfC$ so that $\Col_w(\pi)(\Col_w(D))$ also has a boundary comprised of $s$. It follows that, since the marked point is free, $\gamma_C^* - \gamma_D^* \in T_{(X, \omega)} \cM$. This reduces to the previous case.
\end{proof}

We may assume without loss of generality that $\cM$ is not geminal (by Lemma \ref{L:AddingPPtoGeminal}). Therefore, assume without loss of generality that there is an equivalence class $\bfC_1$ of cylinders on $(X, \omega)$ that does not admit a geminal partition (see Definition \ref{D:Geminal}). As in the proof of Lemma \ref{L:MinIsOptimal}, we may assume without loss of generality that $\bfC_1$ is generic, horizontal, and that its complement contains a generic equivalence class $\bfC_2$. If $\ker(p) \cap T_{(X, \omega)} \cM$ is not contained in $\mathrm{Twist}(\bfC_1, \cM)$, then we may additionally suppose that $\bfC_2$ is involved with rel.

\begin{lem}
If $\cM$ has no rel or if $\ker(p) \cap T_{(X, \omega)} \cM$ is not contained in $\mathrm{Twist}(\bfC_1, \cM)$, then $\cM$ is $h$-geminal up to marked points.
\end{lem}
\begin{proof}
Let $w \in \mathrm{Twist}(\bfC_2, \cM)$ define a typical orthogonal cylinder degeneration (such an element exists by Theorem \ref{T:Typical} and Remark \ref{R:TypicalOrthogonalConstruction}). $\cM_w$ has rank at least two, of minimal homological dimension, and consists of connected surfaces (by Theorem \ref{T:Typical} and Corollary \ref{C:MHD-Boundary:Connected}). Moreover, $\Col_w(\pi)$ is $\cM_w$-optimal (by Lemma \ref{L:MinIsOptimal}). By Assumption \ref{A:IH} and Theorem \ref{T:geminal2} \eqref{I:geminal2:WhatIsM}, $\cM_w$ is either a non-hyperelliptic component of a stratum of Abelian differentials or a full locus of covers of a hyperelliptic locus. In the first case, all the cylinders in $\bfC_1$ would be free, contradicting our assumption that $\bfC_1$ does not admit a geminal partition. Therefore, suppose that $(\cM_w)_{min}$ is a hyperelliptic locus. 

Let $\bfD_1 := \Col_w(\pi)\left(\Col_w(\bfC_1)\right)$. Since $\bfC_1$ does not admit a geminal partition, the same holds for $\pi(\bfC_1)$ and hence for $\bfD_1$ as well, implying that the boundary of one of the cylinders in $\bfD_1$ contains a free marked point (by Corollary \ref{C:HypParallelism}). We are done by Corollary \ref{C:SimpleCylinder} \eqref{I:test:free-pt}.
%
% By Lemma \ref{L:HypParallelism}, $\For(\bfD_1)$ consists of one cylinder fixed or two exchanged by the hyperelliptic involution.  Let $Q$ denote the marked points on $(\Col_w(X, \omega))_{min}$ contained in the interior of a cylinder in $\For(\bfD_1)$. Since $\bfC_1$ cannot be partitioned into geminal subequivalence classes the same must hold for $\pi(\bfC_1)$ and hence for $\bfD_1$ as well. Therefore, the marked points in $Q$ cannot be invariant by the hyperelliptic involution and hence must contain a free point (by Theorem \ref{T:PeriodicPointsHyp}). We are done by Corollary \ref{C:SimpleCylinder}.
\end{proof}

Suppose therefore that $\cM$ has rel and that $\ker(p) \cap T_{(X, \omega)} \cM \subseteq \mathrm{Twist}(\bfC_1, \cM)$. 

\begin{lem}
If $\cM$ is not $h$-geminal up to marked points, then for any typical orthogonal cylinder degeneration $v \in \mathrm{Twist}(\bfC_1, \cM)$,  $(\cM_v)_{min}$ is a genus zero quadratic double. 

Equivalently, for any typical orthogonal cylinder degeneration $w \in \mathrm{Twist}(\pi(\bfC_1), \cM_{min})$,  $(\cM_{min})_w$ is a genus zero quadratic double.
\end{lem}
\begin{proof}
Since $\bfC_1$ is involved with rel by assumption, $\cM_v$ has the same rank as $\cM$ (which is at least two) by Theorem \ref{T:Typical}. So $\cM_v$ is either a component of a stratum of Abelian differentials or a full locus of translation covers of a hyperelliptic locus (by Assumption \ref{A:IH} and Theorem \ref{T:geminal2} \eqref{I:geminal2:WhatIsM}). In the first case, $\cM$ is $h$-geminal up to marked points by Corollary \ref{C:SimpleCylinder} \eqref{I:test:free-cyl}.

Suppose therefore that $(\cM_v)_{min}$ is a hyperelliptic locus, but one which is not a quadratic double. That is, the marked points are not invariant by the hyperelliptic involution and so one of them must be free (by Theorem \ref{T:PeriodicPointsHyp}). Set $\bfD_1 := \Col_v(\pi)(\Col_v(\bfC_1))$

\begin{sublem}
Every free marked point on $\Col_v(\pi)(\Col_v(X, \omega))$ is contained on the boundary of a cylinder in $\bfD_1$.
\end{sublem}
\begin{proof}
By Theorem \ref{T:Typical}, $\mathrm{rank}(\cM) = \mathrm{rank}(\cM_v)$ and $\dim(\cM_v) = \dim(\cM) -1$. Therefore, the dimension $r_v$ of $\ker(p) \cap T_{\Col_v(X, \omega)} \cM_v$ is exactly one less than the dimension $r$ of $\ker(p) \cap T_{(X, \omega)} \cM$. In other words, $r_v = r-1$. Since $\ker(p) \cap T_{(X, \omega)} \cM \subseteq \mathrm{Twist}(\bfC_1, \cM)$, we have, by Theorem \ref{T:CylECTwistSpace},
\[ \dim \mathrm{Twist}(\bfC_1, \cM) = r+1.\]
Since $v$ is typical, $\mathrm{Twist}(\Col_v(\bfC_1), \cM_v)$ has dimension exactly one less than that of $\mathrm{Twist}(\bfC_1, \cM)$ (for instance by Theorem \ref{T:CylinderBoundaryTangentFormula}), i.e.
\[ \dim \mathrm{Twist}(\Col_v(\bfC_1), \cM_v) = r = r_v + 1.\]
Since $\Col_v(X, \omega)$ is connected (by Corollary \ref{C:MHD-Boundary:Connected}), it follows from Theorem \ref{T:CylECTwistSpace} that $\ker(p) \cap T_{\Col_v(X, \omega)} \cM_v \subseteq \mathrm{Twist}(\Col_v(\bfC_1), \cM)$.

Consequently, $\ker(p) \cap T_{\Col_v(\pi)(\Col_v(X, \omega))} (\cM_v)_{min} \subseteq \mathrm{Twist}(\bfD_1, (\cM_v)_{min})$. In particular, every free marked point must be contained on the boundary of a cylinder in $\bfD_1$.
\end{proof}

Since the boundary of a cylinder in $\bfD_1$ contains a free marked point, $\cM$ is $h$-geminal up to marked points by Corollary \ref{C:SimpleCylinder} \eqref{I:test:free-pt}.

\end{proof}

By the preceding lemma, we are left to consider the case where, for any typical orthogonal cylinder degeneration $w \in \mathrm{Twist}(\pi(\bfC_1), \cM_{min})$,  $(\cM_{min})_w$ is a genus zero quadratic double. Note that by Apisa-Wright \cite{ApisaWrightHighRank}, specifically the argument beginning in the paragraph after the proof of Lemma 12.1 through the proof of Corollary 12.3, implies that there are two cylinders $C_1$ and $C_2$ of generically equal circumference in $\pi(\bfC_1)$ so that $\gamma_{\pi(C_1)}^* - \gamma_{\pi(C_2)}^* \in T_{\pi(X, \omega)} \cM_{min}$. Therefore, $\cM$ is $h$-geminal up to marked points by Corollary \ref{C:SimpleCylinder} \eqref{I:test:rel}.

%%%%%%%%%%%%%%%%%%%
%
% SECTION: EXPONENT COMPUTATION
%
%%%%%%%%%%%%%%%%%%%

\section{Proof of Theorem \ref{T:main2}}

This entire section is dedicated to the proof of Theorem \ref{T:main2}. As mentioned in Remark \ref{R:T:main2}, only the forward implication of the claim requires proof. 

Suppose therefore that $\cM$ is an invariant subvariety of rank $r$ that consists of genus $g$ surfaces. Suppose too that, with respect to Lebesgue measure on the unit-area locus of $\cM$, the Kontsevich-Zorich cocycle has $2g-2r$ zero Lyapunov exponents. The goal of this section is to show that $\cM$ is one of the following: the Eierlegende-Wollmilchsau, the Ornithorynque, or full rank. 

When $r = 1$ this result is due to Aulicino-Norton \cite{AulicinoNorton}, so we will suppose in this section that $r > 1$. By Forni's criterion (see \cite{F} and \cite{Forni:Criterion}), $\cM$ must have minimal homological dimension. Therefore, by Theorem \ref{T:main}, after marking all periodic points (which does not affect the number of zero exponents) $\cM$ is $h$-geminal. We will suppose in order to deduce a contradiction that $\cM$ is not full rank. 

By Theorem \ref{T:main}, there is a hyperelliptic locus $\cM_{min}$ so that $\cM$ is a full locus of covers of $\cM_{min}$. Since we have supposed that $\cM$ is not full rank, the degree of these covers is at least two. Since $\cM_{min}$ is a hyperelliptic locus, $\cM_{min}$ is the locus of holonomy double covers of surfaces in the stratum of quadratic differentials $\cQ(d_1, \hdots, d_n)$ where $\sum_i d_i = -4$. By Eskin-Kontsevich-Zorich \cite[Corollary 1]{EKZbig} the sum of positive Lyapunov exponents of $\cM_{min}$ (and indeed for any invariant subvariety therein) is $\frac{1}{4} \sum_{d_i odd} \frac{1}{d_i + 2}$. 

\begin{lem}\label{L:Nonvarying}
Let $\cN \subseteq \cM$ be any invariant subvariety. The sum of the positive Lyapunov exponents for the Kontsevich-Zorich cocycle with respect to Lebesgue measure on the unit-area locus of $\cN$ is $\frac{1}{4} \sum_{d_i odd} \frac{1}{d_i + 2}$. 
\end{lem}
\begin{proof}
% By Avila-Eskin-M\"oller \cite[Theorem 1.5]{AEM}, the Hodge bundle $H_{\mathbb{R}}^1$ over $\cM$ is semisimple in the sense that any flat bundle has a complementary flat bundle. We will let $W$ denote the complement of $p(T\cM)$ in $H_{\mathbb{R}}^1$ over $\cM$, i.e. 
% \[ H_{\mathbb{R}}^1 = p(T\cM) \oplus W.\]
Since $p(T\cM)$ is complex-symplectic (by Avila-Eskin-M\"oller \cite{AEM}) letting $W$ be the subbundle of $H_{\mathbb{R}}^1 \restriction_{\cM}$ that is symplectically orthogonal to $p(T\cM)$, implies that
\[ H_{\mathbb{R}}^1\restriction_{\cM} = p(T\cM) \oplus W.\]
The sum of the positive Lyapunov exponents of $p(T\cM)$ restricted to $\cN$ is $\frac{1}{4} \sum_{d_i odd} \frac{1}{d_i + 2}$ by Eskin-Kontsevich-Zorich \cite[Corollary 1]{EKZbig}. Therefore, it suffices to show that $W$ restricted to $\cN$ belongs to the Forni subspace. This will follow from showing that the monodromy of $W$ over $\cM$ belongs to a compact group (this property is then inherited by the restriction of $W$ to $\cN$). This in turn follows by Filip \cite[Theorem 1.1]{FiZero}.
%%%%%%%%%%%%%%%%%%%%%
%
% DO NOT DELETE
% MORE DETAILS
%
%%%%%%%%%%%%%%%%%%%%%
% In Section 1.3 of Filip, Filip notes that Theorem 1.1 implies the conjecture of Forni-Matheus-Zorich on zero exponents except when the monodromy of (in our language) W is contained in SO*(2n)). SO*(2n) is a synonym for SO(n, H), which (see Dave Witte Morris Appendix A.2 in his book on Arithmetic Groups) is noncompact if and only if n > 2. So SO*(2n) acts by the standard representation it acts on a vector space of type R^{4n}. By Filip it has zero Lyapunov exponents only when n is odd in which case it has exactly four. This implies (again in our special case where W has only zero exponents) that the monodromy of W is not by SO*(2n). Therefore, it is by an exterior rep of SU(p,q) which is a case where the FMZ conjecture holds. In our special case, the FMZ conjecture implies that the monodromy of W is by a compact group. 
\end{proof}

% The same is true for the sum of positive Lyapunov exponents in the invariant subbundle $p\left( T \cM \right)$.  

% Let $\cM_{min}$ denote the orbit closure of $(X_{min}, \omega_{min})$ where $(X, \omega)$ is a surface with dense $\mathrm{GL}(2, \mathbb{R})$ orbit in $\cM$. Since $\cM_{min}$ is a hyperelliptic locus (by Theorem \ref{T:main}) suppose that it is the locus of holonomy double covers of a surface in the stratum of quadratic differentials $\cQ(d_1, \hdots, d_n)$ where $\sum_i d_i = -4$. By Eskin-Kontsevich-Zorich \cite[Corollary 1]{EKZbig} the sum of positive Lyapunov exponents of $\cM_{min}$ (and indeed for any invariant subvariety therein) is $\frac{1}{4} \sum_{d_i odd} \frac{1}{d_i + 2}$. The same is true for the sum of positive Lyapunov exponents in the invariant subbundle $p\left( T \cM \right)$.  

% Suppose to a contradiction that $p \left( T \cM \right)$ is an invariant subbundle that contains all nonzero Lyapunov exponents. Since zero Lyapunov exponents arise from monodromy (see Filip), this property holds for any invariant subvariety contained in $\cM$. 

Let $(m_i)$ (resp. $(m_i')$) denote the order of the zeros on surfaces in $\cM$ (resp. $\cM_{min}$). 

Let $(X', \omega')$ be any square-tiled surface in $\cM_{min}$ so that there is a surface $(X, \omega)$ in $\cM$ that covers $(X', \omega')$ under its $\cM$-optimal map
\[ \pi_{opt}: (X, \omega) \ra (X', \omega').\] Note that $(X, \omega)$ is also square-tiled. Let $\Gamma$ (resp. $\Gamma'$) denote the Veech group of $(X, \omega)$ (resp. $(X', \omega')$). Notice that $\Gamma$ is a subgroup of $\Gamma'$. Given an element $g_j\Gamma$ of $ \mathrm{SL}\left(2, \mathbb{Z}\right)/\Gamma$, we will let $(h_{ij})$ and $(w_{ij})$ denote the heights and widths of horizontal cylinders on $g_j \cdot (X, \omega)$. We will let $(h_{ik}')$ and $(w_{ik}')$ denote similar quantities given $g_k\Gamma'$ applied to $(X', \omega')$. By Eskin-Kontsevich-Zorich \cite[Corollary 8]{EKZbig} the sum of the positive Lyapunov exponents over the Teichm\"uller curve generated by $(X, \omega)$ is the following,
\begin{equation}\label{E:SqSum}  \frac{1}{12} \sum_i \frac{m_i(m_i+2)}{m_i+1} + \frac{1}{| \mathrm{SL}\left(2, \mathbb{Z}\right)/\Gamma|} \sum_{g_j \in \mathrm{SL}\left(2, \mathbb{Z}\right)/\Gamma} \sum_i \frac{h_{ij}}{w_{ij}}.   \end{equation}
Let $d$ denote the degree of $\pi_{opt}$. By definition of an $\cM$-optimal map, the preimage of a cylinder on $(X', \omega')$ on $(X, \omega)$ is a single cylinder of the same height and whose circumference is $d$ times longer. This observation implies the following,
\begin{equation}\label{E:SqSumSummand1}  \frac{1}{| \mathrm{SL}\left(2, \mathbb{Z}\right)/\Gamma|} \sum_{g_j \in \mathrm{SL}\left(2, \mathbb{Z}\right)/\Gamma} \sum_i \frac{h_{ij}}{w_{ij}} = \frac{1}{|\mathrm{SL}\left(2, \mathbb{Z}\right)/\Gamma'|} \frac{1}{| \Gamma'/\Gamma|} \sum_{g_j \in \Gamma'/\Gamma} \sum_{g_k \in  \mathrm{SL}\left(2, \mathbb{Z}\right)/\Gamma'} \sum_i \frac{h_{ik}'}{dw_{ik}'}. 
\end{equation}
Since the right-hand side of Equation \eqref{E:SqSumSummand1} does not depend on $j$, we have the following,
\begin{equation}\label{E:SqSumSummand2} 
 \frac{1}{| \mathrm{SL}\left(2, \mathbb{Z}\right)/\Gamma|} \sum_{g_j \in \mathrm{SL}\left(2, \mathbb{Z}\right)/\Gamma} \sum_i \frac{h_{ij}}{w_{ij}} =  \frac{1}{d|\mathrm{SL}\left(2, \mathbb{Z}\right)/\Gamma'|} \sum_{g_k \in \mathrm{SL}\left(2, \mathbb{Z}\right)/\Gamma'} \sum_i \frac{h_{ik}'}{w_{ik}'}. \end{equation}
% Notice that the second summand of \eqref{E:SqSum} is equal to the following,
% \begin{equation} \frac{1}{|\Gamma \backslash \Gamma'|} \frac{1}{|\Gamma' \backslash \mathrm{SL}\left(2, \mathbb{Z}\right)|} \sum_{g_j \in \Gamma \backslash \Gamma'} \sum_{g_k \in \Gamma' \backslash \mathrm{SL}\left(2, \mathbb{Z}\right)} \sum_i \frac{h_{ik}'}{dw_{ik}'} = \frac{1}{d|\Gamma' \backslash \mathrm{SL}\left(2, \mathbb{Z}\right)|} \sum_{g_k \in \Gamma' \backslash \mathrm{SL}\left(2, \mathbb{Z}\right)} \sum_i \frac{h_{ik}'}{w_{ik}'} \end{equation}
%This is because given an element $g_jg_k$ in $\Gamma\backslash \mathrm{SL}\left(2, \mathbb{Z}\right)$ where $g_j \in \Gamma \backslash \Gamma'$ and $g_k \in \Gamma' \backslash \mathrm{SL}\left(2, \mathbb{Z}\right)$ the surface being covered only depends on $g_k$ and since $\pi_{opt}$ is a good map the preimage of any horizontal cylinder on the surface being covered is exactly one cylinder of the same height and whose circumference is $d$ times as long.
Combining Eskin-Kontsevich-Zorich \cite[Corollary 1]{EKZbig} and \cite[Corollary 8]{EKZbig} yields,
\begin{equation}\label{E:EKZ} \frac{1}{4} \sum_{d_i odd} \frac{1}{d_i + 2} = \frac{1}{12} \sum_i \frac{m_i'(m_i'+2)}{m_i'+1} + \frac{1}{| \mathrm{SL}\left(2, \mathbb{Z}\right)/\Gamma'|} \sum_{g_k \in \mathrm{SL}\left(2, \mathbb{Z}\right)/\Gamma'} \sum_i \frac{h_{ik}'}{w_{ik}'}. \end{equation}
where the final sum is over horizontal cylinders (indexed by $i$) on $g_k (X', \omega')$. Combining Equations \eqref{E:EKZ} and \eqref{E:SqSumSummand2} yields,
\begin{equation}\label{E:SqSumSummand3} 
 \frac{1}{| \mathrm{SL}\left(2, \mathbb{Z}\right)/\Gamma|} \sum_{g_j \in \mathrm{SL}\left(2, \mathbb{Z}\right)/\Gamma} \sum_i \frac{h_{ij}}{w_{ij}} =  \frac{1}{d} \left( \frac{1}{4} \sum_{d_i odd} \frac{1}{d_i + 2} - \frac{1}{12} \sum_i \frac{m_i'(m_i'+2)}{m_i'+1} \right). \end{equation}
Substituting Equation \eqref{E:SqSumSummand3} into Equation \eqref{E:SqSum} and appealing to Lemma \ref{L:Nonvarying} yields the following,
% \begin{equation} \frac{1}{|\Gamma \backslash \Gamma'|} \frac{1}{|\Gamma' \backslash \mathrm{SL}\left(2, \mathbb{Z}\right)|} \sum_{g_j \in \Gamma \backslash \Gamma'} \sum_{g_k \in \Gamma' \backslash \mathrm{SL}\left(2, \mathbb{Z}\right)} \sum_i \frac{h_{ik}'}{dw_{ik}'} = \frac{1}{d|\Gamma' \backslash \mathrm{SL}\left(2, \mathbb{Z}\right)|} \sum_{g_k \in \Gamma' \backslash \mathrm{SL}\left(2, \mathbb{Z}\right)} \sum_i \frac{h_{ik}'}{w_{ik}'} \end{equation}
% \begin{equation}\label{E:MainEquality}  \frac{1}{12} \sum_i \frac{m_i(m_i+2)}{m_i+1} + \frac{1}{d} \left( \frac{1}{4} \sum_{d_i odd} \frac{1}{d_i + 2} - \frac{1}{12} \sum_i \frac{m_i'(m_i'+2)}{m_i'+1} \right) \end{equation}
% Since this expression is equal to $\frac{1}{4} \sum_{d_i odd} \frac{1}{d_i + 2}$ we have the following equality, 
\begin{equation}\label{E:MainEquality} \frac{1}{4}\left( 1 - \frac{1}{d} \right)\sum_{d_i odd} \frac{1}{d_i + 2} = \frac{1}{12} \left( \sum_i \frac{m_i(m_i+2)}{m_i+1} - \frac{1}{d} \sum_i \frac{m_i'(m_i'+2)}{m_i'+1}\right). \end{equation}
Let $g$ be the genus of $(X, \omega)$ and $g'$ the genus of $(X', \omega')$. Using the trivial observations that $1 < \frac{m_i + 2}{m_i + 1} \leq \frac{3}{2}$ (since $m_i \geq 1$) and that $\sum_i m_i = 2g-2$ (and that the same relations hold after adding primes), Equation \eqref{E:MainEquality} produces the following inequality,  
\begin{equation}\label{E:MainInequality1} 3\left( 1 - \frac{1}{d} \right)\sum_{d_i odd} \frac{1}{d_i + 2} > (2g-2) - \frac{3}{2d}(2g'-2). \end{equation}
 By the Riemann-Hurwitz formula, Equation \eqref{E:MainInequality1} becomes,
\begin{equation}\label{E:MainInequality2} 3 \left( 1 - \frac{1}{d} \right)\sum_{d_i odd} \frac{1}{d_i + 2} > \left(d - \frac{3}{2d} \right) (2g'-2) + R. \end{equation}
where $R$ is the ramification term, i.e. $R = \sum_{p \in X} (e_p-1)$ where $e_p$ is the ramification index of $\pi_{opt}$ at $p$. Since $\cM$ is $h$-geminal of rank at least two and since $\pi_{opt}$ has degree at least two, $\pi_{opt}$ is branched over every regular Weierstrass point except possibly one (by Corollary \ref{C:h-geminal-branching}). This implies that $R \geq P-1$ where $P$ denotes the number of regular Weierstrass points on $(X', \omega')$. Since the sum $\sum_{d_i odd} \frac{1}{d_i + 2}$ can be interpreted as a sum over the Weierstrass points of $(X', \omega')$ - where each regular Weierstrass point contributes $1$ to the sum and every other Weierstrass point contributes a number bounded above by $\frac{1}{3}$ - Equation \eqref{E:MainInequality2} implies the following,
\begin{equation} 3 \left( 1 - \frac{1}{d} \right)\left( P +\frac{1}{3}\left(2g'+2-P\right)  \right) > \left(d - \frac{3}{2d} \right) (2g'-2) + (P-1). \end{equation}
Simplifying the left hand side gives,
\[ \left( 1 - \frac{1}{d} \right)\left( \left(2g'-2\right) + 4 + 2P \right) > \left(d - \frac{3}{2d} \right) (2g'-2) + (P-1).   \]
Multiplying through by $2d$ yields,
\[ \left( d - 1 \right)\left( 2\left(2g'-2\right) + 8 + 4P \right) > \left(2d^2 - 3 \right) (2g'-2) + 2d(P-1).   \]
Combining terms produces the following,
\[ 8(d-1)+ 4P(d-1) - 2d(P-1) > \left(2d^2 - 2d - 1 \right) (2g'-2).\]
So we have,
\[ 10d - 8 + (2d -4)P    > \left(2d^2 - 2d - 1 \right) (2g'-2).\]
Since $P \leq 2g'+2 = (2g'-2)+4$,
\begin{equation}\label{E:MainInequality3} 18d - 24 > \left( 2d^2 - 4d + 3 \right)(2g'-2) \end{equation}
This shows that $d \leq 4$ and that $g' = 2$ (notice that $g' \geq 1$ and that $g' \ne 1$ since then $\cM$ would be a locus of torus covers, contradicting the fact that it has rank at least two). Set
\[ \Lambda := 3\left( 1 - \frac{1}{d} \right)\sum_{d_i odd} \frac{1}{d_i + 2} \]
Since $g' = 2$, $\cM_{min}$ is the locus of holonomy double covers of either $\cQ(1, -1^5)$ or $\cQ(2, -1^6)$, for which $P = 5$ and $P = 6$ respectively. Therefore, $\Lambda = 16\left( 1 - \frac{1}{d} \right)$ when $P = 5$ and $\Lambda = 18\left( 1 - \frac{1}{d} \right)$ when $P = 6$. Since $g' = 2$, the inequality $P-1 \leq R$ together with Equation \eqref{E:MainInequality2} implies the following
\begin{equation}\label{E:MainInequality4} P-1 \leq R < \Lambda - \left( 2d - \frac{3}{d} \right). \end{equation}
We will now divide into two cases. 

\noindent \textbf{Case 1: $P = 6$.}

In this case, Equation \eqref{E:MainInequality4} becomes
\[ 5 \leq R < 18\left( 1 - \frac{1}{d} \right) - \left( 2d - \frac{3}{d} \right). \]
Since $2 \leq d \leq 4$, $R \in \{5, 6\}$. Since $R$ is even, $R = 6$. Since there is branching over at least $5$ regular Weierstrass points, there are at least four regular Weierstrass points that are simple branch points, i.e. their preimage contains exactly one ramification point and that point is doubly ramified. If none of the zeros of $\omega'$ are branch points, then $(X, \omega)$ has at least $2d+4$ zeros of order $1$. Similarly, if one of the zeros is a branch point then every branch point is simple and so $(X, \omega)$ has $2d+4$ zeros of order $1$. Equation \eqref{E:MainEquality} implies that
\[ 18\left( 1 - \frac{1}{d} \right) \geq (4+2d)\frac{3}{2}  - \frac{3}{d}.\]
This equation is equivalent to 
\[ 0 > 6d^2 - 24d + 30 \]
which never holds for $d$ real, a contradiction. 

\noindent \textbf{Case 2: $P = 5$.}

In this case, Equation \eqref{E:MainInequality4} becomes\[ 4 \leq R \leq 16\left( 1 - \frac{1}{d} \right) - \left( 2d - \frac{3}{d} \right). \]
Since $2 \leq d \leq 4$, $R \in \{4, 5\}$. Since $R$ is even, $R = 4$. Since there is branching over at least $4$ regular Weierstrass points, the only branch points are $4$ regular Weierstrass points, which are simple branch points. This implies that $(X, \omega) \in \cH\left( 1^4, 2^d \right)$ where exponents indicate multiplicity. Equation \eqref{E:MainEquality} then implies that
\[ 16\left( 1 - \frac{1}{d} \right) = 4\left( \frac{3}{2} \right) + d \left( \frac{8}{3}\right) - \frac{8}{3d}. \]
So
\[8d^2 - 30d + 40 = 0. \]
However, this equation has no real roots, so we have a contradiction. 

\bibliography{mybib}{}
\bibliographystyle{amsalpha}
\end{document}